\documentclass[reqno,a4paper]{amsart}
\usepackage{graphicx}
\usepackage{amsmath}
\usepackage{mathrsfs}
\usepackage{amsfonts}
\usepackage{amssymb}
\usepackage{enumerate}
\usepackage{latexsym}
\usepackage{graphpap}
\usepackage{cite}
\usepackage[lined,boxed,commentsnumbered,ruled]{algorithm2e}
\usepackage{tikz}
\usepackage{multirow}
\usepackage{makecell}

\allowdisplaybreaks

\newtheorem{theorem}{Theorem}[section]
\newtheorem{corollary}[theorem]{Corollary}
\newtheorem{lemma}[theorem]{Lemma}
\newtheorem{proposition}[theorem]{Proposition}

\theoremstyle{definition}

\newtheorem{remark}[theorem]{Remark}

\newtheorem{algo}{Algorithm}

\theoremstyle{remark}
\numberwithin{equation}{section}

\newcommand{\bfa}{\mathbf{a}}
\newcommand{\bfb}{\mathbf{b}}

\newcommand{\bfs}{\mathbf{s}}
\newcommand{\bft}{\mathbf{t}}
\newcommand{\bfu}{\mathbf{u}}
\newcommand{\bfv}{\mathbf{v}}
\newcommand{\bfw}{\mathbf{w}}

\newcommand{\s}{\operatorname{\mathsf{mSt}}}

\newcommand{\sn}{\operatorname{\mathsf{mSt}}_n}
\newcommand{\rst}{\operatorname{\mathsf{rSt}}}
\newcommand{\lst}{\operatorname{\mathsf{lSt}}}

\newcommand{\ev}{\operatorname{\mathsf{ev}}}

\newcommand{\fp}{\operatorname{\mathsf{fp}}}
\newcommand{\ip}{\operatorname{\mathsf{ip}}}
\newcommand{\m}{\operatorname{\mathsf{m}}}

\newcommand{\op}{^*}
\newcommand{\var}{\mathsf{Var}}
\newcommand{\fb}{finitely based}
\newcommand{\nfb}{non-finitely based}
\newcommand{\inmon}{involution monoid}
\newcommand{\insem}{involution semigroup}

\newcommand{\con}{\operatorname{\mathsf{con}}}

\newcommand{\occ}{\operatorname{\mathsf{occ}}}

\begin{document}
\begin{sloppypar}
\title[Representations and identities of $\sn$ with involution]{Representations and identities of involution Plactic-like monoids arising from the meet of the stalactic congruence and its dual}%
\thanks{This research was partially supported by the National Natural Science Foundation of China (Nos.
12401017, 12271224, 12171213), the Fundamental Research Funds for the Central University (No. lzujbky-2023-ey06) and the Natural Science Foundation of Gansu Province (No. 23JRRA1055).}

\author[B.B.Han]{Bin Bin Han}
\address{Aliyun School of Big Data, Changzhou University,
  Changzhou, Jiangsu, 213164, P.R. China,
   School of Mathematics and Statistics, Lanzhou University, Lanzhou, Gansu 730000, P.R. China}
\email{hanbb24@cczu.edu.cn}
\author[W.T.Zhang]{Wen Ting Zhang$^\star$} \thanks{$^\star$Corresponding author}
\address{School of Mathematics and Statistics, Lanzhou University, Lanzhou, Gansu 730000, P.R. China}
\email{zhangwt@lzu.edu.cn}
\author[Y. F. Luo]{Yan Feng Luo}
\address{School of Mathematics and Statistics, Lanzhou University, Lanzhou, Gansu 730000, P.R. China}
\email{luoyf@lzu.edu.cn}



\begin{abstract}
Let $\sn$ be the plactic-like monoid obtained by factoring the free monoid  over a finite alphabet $\mathcal{A}_n$ by the meet of the stalactic congruence and its dual. In this paper, we prove that $\sn$ can be equipped with multiple involutions, and divide these involutions into $\lfloor\frac{n}{2}\rfloor+1$ types.  A faithful representation of $\sn$ under each of these involutions is obtained.  We give transparent combinatorial characterizations of identities for $\sn$ under each  involution, and so the finite basis problem and identity checking problem for them are solved.
\end{abstract}

\maketitle

\section{Introduction}
The plactic monoid, whose elements can be identified with Young
tableaux, is famous for its interesting connection with combinatorics \cite{Lot02} and applications in symmetric functions \cite{Mac08}, representation theory \cite{Ful97} and Kostka-Foulkes polynomials \cite{LS78,LS81}. Plactic-like monoids, whose elements can be identified with combinatorial objects as the plactic monoid, have also been extensively studied, such as the hypoplactic monoid \cite{KT97,Nov00}, the sylvester and \#-sylvester monoids \cite{HNT05}, the Baxter monoid \cite{Gir12}, the taiga monoid \cite{Pri13}, the stalactic monoid \cite{HNT07,Pri13} and the stylic monoid \cite{AR22}.
In the research of plactic-like monoids, their representations and identities have received much attention. The plactic monoids of finite rank were shown to have faithful tropical representations by Johnson and Kambites \cite{JK19}.
The hypoplactic, stalactic, taiga, sylvester and Baxter monoids of each finite rank were shown to have faithful tropical representations \cite{CJKM21}.
It is shown that all hypoplactic [resp. Baxter, sylvester, stalactic and taiga] monoids of rank greater than or equal to $2$ generate the same variety and are finitely based \cite{CMR21b, CJKM21,HZ,CMR21a}.
Aird and Ribeiro have given a faithful representation of the stylic monoids of each finite rank and solved the finite basis problems for them \cite{TD22}.
Volkov has solved the finite basis problem for the stylic monoid by different means \cite{volkov2022}.

Representations and identities can also be considered for their involution cases. Recall that an \textit{involution semigroup} $(S,\op)$ is a unary semigroup  that satisfies the identities
\begin{equation}
(x^*)^* \approx  x \quad \text{and} \quad (xy)^* \approx y^*x^*; \label{id: inv}
\end{equation}
the unary operation $^*$ is an involution of $S$, and $S$ is called the \textit{semigroup reduct} of $(S,\op)$.
For example, multiplicative matrix semigroup $(M_n,^D)$ over any field with skew transposition~$^D$ (that is,  $(A^{D})_{ij}=A_{(n+1-j)(n+1-i)}$) is an involution semigroup.
Refer to \cite{Lee20,ADPV14,ADV12a,ADV12b} for more information on the identities of involution semigroups.
Recently, Aird and Ribeiro have given a faithful representation of the stylic monoid with involution of each finite rank and solved the finite basis problems for them \cite{TD22}; Han et al. have given faithful representations of the hypoplactic and Baxter monoids with involution of each finite rank and solved the finite basis problems for them \cite{HZLZa, HZLZb}.

Recall that a monoid congruence is an equivalence relation which is
compatible with the product. Using two congruences, one can build two new congruences, that is,
their meet and their join; further one can obtain two new monoids by factoring the free monoid
over an alphabet by these two congruences.
For example, the Baxter congruence is the meet of the sylvester congruence and its dual\cite{Gir12}, and
the join of the sylvester congruence and its dual gives  the hypoplactic  congruence \cite{Nov00}.

Recently, Aird and Ribeiro studied the combinatorial and identity properties of the plactic-like monoid $\sn$  obtained by factoring the free monoid  over a finite alphabet $\mathcal{A}_n$ by the meet of the stalactic congruence and its dual\cite{AR24}. It was shown that the identity checking problem for $\sn$ is in the complexity class $\mathsf{P}$ and the varieties generated by $\sn$ for all $n\geq2$ are finitely based which can be defined by the identities
\begin{gather}
xhxykx\approx xhyxkx,  \quad xhyxky \approx xhxyky. \label{id:xhxyty}
\end{gather}

In this paper, we investigate the representations and identities of $\sn$ with involution.
This paper is organized as follows. Notation and background information of the paper
are given in Section \ref{sec: prelim}. In particular, we prove that $\sn$  has $\lfloor\frac{n}{2}\rfloor+1$ types of involutions, denoted by $^{\sigma_i}$ with $0\leq i \leq \lfloor\frac{n}{2}\rfloor$. In Section 3, we exhibit a faithful representation of $\sn$ under each involution. In Sections 4, 5 and 6, we give transparent combinatorial characterizations of the identities satisfied by them, respectively; based on these results, we solve the finite basis problem and the identity checking problem for $\sn$ under each involution. The properties of $\sn$ with involution  $^{\sigma_i}$ are summarized in Table~\ref{table:compare}.
{\small
\begin{table*}[t]
\centering
\caption{Some properties of $\sn$ under each involution}
\begin{tabular}{|c|c|c|c|c|c|c|c|c|c|c|c|c|}
\hline
Properties  &{\makecell[c]{$(\s_n,^{\sigma_0})$\\ $n\geq 2$}} & {\makecell[c]{$(\sn,^{\sigma_1})$ \\$n\geq 2$}} &{\makecell[l]{$(\sn,^{\sigma_i})$\\ $n\geq 4, 2\leq i\leq \lfloor\frac{n}{2}\rfloor$}} \\
\hline
{\makecell[c]{Matrix \\representations}}  & {\makecell[l]{Theorem 3.3}} &  {\makecell[c]{Theorem 3.3}} & Theorem 3.3 \\
\hline
{\makecell[c]{Characterization \\of its identities}}    & Theorem\,\ref{lem:inv1ids2} &  {\makecell[l]{Theorems\,\ref{lem:inv2ids2} and \ref{lem:inv2ids3}}} & Theorem\,\ref{lem:inv3ids4} \\
\hline
{\makecell[c]{Finite basis \\property}}	 & Finitely based &  Finitely based & Finitely based \\
\hline
{\makecell[c]{Identity \\checking problem}}   & $\mathsf{P}$ & $\mathsf{P}$ & $\mathsf{P}$ \\
\hline
\end{tabular}
\label{table:compare}
\end{table*}}

\section{Preliminaries}\label{sec: prelim}
Most of the notation and definition of this article are given in this section.
Refer to the monograph of Burris and Sankappanavar~\cite{BS81} for any undefined notation and terminology of universal algebra in general.	

\subsection{Words and identities}	
Let~$\mathcal{X}$ be a nonempty alphabet and $\mathcal{X}^* =\{x^* \,|\, x \in \mathcal{X}\}$ be a disjoint copy of~$\mathcal{X}$.
Elements of $\mathcal{X} \cup \mathcal{X}^*$ are called \textit{variables}, elements of the free involution monoid $F_{\mathsf{inv}}^{\varepsilon}(\mathcal{X})=(\mathcal{X} \cup \mathcal{X}^*)^+ \cup \{\varepsilon\}$ are called \textit{words}, and elements of the free monoid $\mathcal{X}^{\star}=\mathcal{X}^+ \cup \{\varepsilon\}$ are called \textit{plain words}. A word $\bfu$ is a \textit{factor} of a word $\bfw$ if $\bfw = \bfa\bfu\bfb$ for some $\bfa, \bfb \in F_{\mathsf{inv}}^{\varepsilon}(\mathcal{X})$.

Let $\bfu\in F_{\mathsf{inv}}^{\varepsilon}(\mathcal{X})$ be a word and $x, y \in \mathcal{X} \cup \mathcal{X}^*$ be distinct variables. The \textit{content} $\con (\bfu)$ of $\bfu$ is the set of variables that occur in~$\bfu$.  Denote by $\occ(x, \bfu)$ the number of occurrences of $x$ in $\bfu$. The \textit{plain projection} $\overline{\bfu}$ of $\bfu$ is the plain word obtained from $\bfu$ by removing all occurrences of the symbol $^*$.
The \textit{length} $|\bfu|$ of $\bfu$  is the number of variables occurring in $\bfu$.
The \textit{initial part} [resp. \textit{final part}] of $\bfu$, denoted by $\ip(\bfu)$ [resp. $\fp(\bfu)$], is the word obtained from $\bfu$ by retaining the occurrence of each variable $x\in\con(\bfu)$ satisfying $x\not\in\con(\bfu_1)$ [resp. $x\not\in\con(\bfu_2)$] where $\bfu=\bfu_1x\bfu_2$.
For any $x_1, x_2, \ldots, x_n \in \mathcal{X}\cup\mathcal{X}^*$ such that $\overline{x_1}, \overline{x_2}, \ldots, \overline{x_n}\in \mathcal{X}$ are distinct variables, let $\bfu[x_1, x_2, \ldots, x_n]$ denote the word obtained from $\bfu$ by retaining only the variables $x_1, x_1^*, x_2, x_2^*, \ldots, x_n, x_n^*$.


%
An \textit{identity} [resp. \textit{plain identity}] is an expression $\bfu \approx \bfv$, where $\bfu$ and $\bfv$ are words in $(\mathcal{X} \cup \mathcal{X}^*)^+$ [resp. $\mathcal{X}^+$]. The identity $\overline{\bfu}\approx \overline{\bfv}$ is called the \textit{plain projection} of $\bfu\approx \bfv$.
We write $\bfu = \bfv$ if $\bfu$ and $\bfv$ are identical.  An identity $\bfu \approx \bfv$ is \textit{non-trivial} if $\bfu \neq \bfv$.
An identity $\bfu \approx \bfv$ is \textit{balanced} if $\mathsf{occ}(x, \bfu)=\mathsf{occ}(x, \bfv)$ for any $x\in \mathsf{con}(\bfu\bfv)$.

An  {\insem} $(S,\op)$ \textit{satisfies} an identity $\bfs \approx \bft$, if for any substitution $\varphi: \mathcal{X} \to S$, the elements $\varphi(\bfs)$ and $\varphi(\bft)$ of $S$ coincide; in this case, $\bfs \approx \bft$ is also said to be an \textit{identity of}  $(S,\op)$. For any {\insem} $(S,\op)$,
a set $\Sigma$ of identities of $(S,\op)$ is an \textit{identity basis} for $(S,\op)$ if every identity of $(S,\op)$ is deducible from $\Sigma$.
An {\insem} is \textit{\fb} if it has some finite identity basis; otherwise, it is \textit{\nfb}. The \textit{finite basis problem} for involution semigroups is the decision problem of determining whether a given semigroup is finitely based. The \textit{identity checking problem} \textsc{Check-Id}$(S,\op)$ for an involution semigroup $(S,\op)$ is the decision problem whose
instance is an arbitrary identity $\bfu \approx \bfv$, and the answer to such an instance is `YES'
if $S$ satisfies $\bfu \approx \bfv$, and `NO' if it does not.

\begin{remark}\label{rem:delete}
Note that assigning the unit element to a variable~$x$ in an identity is effectively the same as removing all occurrences of~$x$ and~$x^*$. Therefore any {\inmon} that satisfies an identity $\bfs \approx \bft$ also satisfies the identity $\bfs [x_1, x_2, \ldots, x_n] \approx \bft[x_1, x_2, \ldots, x_n]$ for any distinct variables $\overline{x_1}, \overline{x_2}, \ldots, \overline{x_n} \in \mathcal{X}$.
\end{remark}


\subsection{The left-stalactic and right-stalactic monoids}
Let $\mathcal{A}= \{ 1 < 2 < 3 < \cdots \}$ denote the ordered infinite alphabet  and $\mathcal{A}_n= \{ 1 < 2 < 3 < \cdots <n\}$ the ordered finite alphabet.
The combinatorial objects and insertion algorithms related to the left-stalactic and right-stalactic monoids are given in the following.

A \textit{stalactic tableau} is a finite array of variables of $\mathcal{A}$ in which columns are top-aligned, and two variables appear in the same column if and only if they are equal.

\begin{algo}[{\cite[Section 3.7]{HNT07}}]\label{algo:stal}
Input a stalactic tableau $T$ and a variable $a \in \mathcal{A}$.
If $a$ does not appear in $T$, add $a$ to the left of the top row of $T$; if $a$ does appear in $T$, add $a$ to the bottom of the column in which $a$ appears. Output the new tableau.
\end{algo}

Let $a_1, a_2, \dots, a_k\in \mathcal{A}$ and $\bfw=a_1a_2\cdots a_k \in \mathcal{A}^{\star}$. Then the stalactic tableau ${\rm P}^\leftarrow_{\rst}(\bfw)$ of $\bfw$ is obtained as follows: reading
$\bfw$ from right-to-left, one starts with an empty tableau and inserts each variable in $\bfw$ into a stalactic tableau according to  Algorithm \ref{algo:stal}.
For example, ${\rm P}^\leftarrow_{\rst}(36131512665)$ is given as follows:
\begin{equation*}
\begin{tikzpicture}
[ampersand replacement=\&,row sep=-\pgflinewidth,column sep=-\pgflinewidth]
\matrix [nodes=draw]
{
\node {3}; \& \node{1}; \& \node{2};\& \node{6}; \& \node{5}; \\
\node {3}; \& \node{1}; \&  \& \node{6};  \& \node{5};\\
 \& \node{1}; \& \& \node{6}; \&\\
};
\end{tikzpicture}
\end{equation*}

Define the relation $\equiv_\mathsf{r}$ by
\[
\bfu \equiv_\mathsf{r} \bfv \quad\text{if and only if} \quad{\rm P}^\leftarrow_\mathsf{rSt}(\bfu) =  {\rm P}^\leftarrow_\mathsf{rSt} (\bfv)
\]
where $\bfu$ and $\bfv$ are words over
alphabet $\mathcal{A}$. It turns out that the relation $\equiv_\mathsf{r}$ is a congruence on $\mathcal{A}^{\star}$ which is called  the \textit{right-stalactic congruence}. The infinite-rank \textit{right-stalactic monoid} $\rst_{\infty}$ is the factor monoid $ \mathcal{A}^{\star}/_{\equiv_\mathsf{r}}$ and the  \textit{right-stalactic monoid} of rank $n$, denoted by $\rst_n$, is the factor monoid $ \mathcal{A}_n^{\star}/_{\equiv_\mathsf{r}}$.


A new algorithm by can be obtained by replacing ``left" with
``right" in Algorithm 1.
Using this new algorithm, one can compute a unique stalactic tableau from
a word $\bfw=a_1a_2\cdots a_k \in \mathcal{A}^{\star}$: starting from the empty tableau, read $\bfw$ from left-to-right and insert its variables one-by-one into the tableau. The resulting tableau is denoted by ${\rm P}^\rightarrow_{\lst}(\bfw)$. Thus the \textit{left-stalactic congruence} $\equiv_\mathsf{l}$, the infinite-rank \textit{left-stalactic monoid} $\mathsf{lSt_{\infty}}$ and \textit{left-stalactic monoid} of rank $n$ $\mathsf{lSt_{n}}$ are defined in a similar manner with the right-stalactic congruence and the right-stalactic monoid.


For any $\bfu \in \mathcal{A}^{\star}$, the \textit{evaluation} of $\bfu$, denoted by $\ev(\bfu)$, is the
tuple of non-negative integers, indexed by $\mathcal{A}$, whose $a$-th element is $\occ(a, \bfu)$, thus this tuple describes the number of each variable in $\mathcal{A}$ that appears in $\bfu$. It is immediate from the definition of the monoid $\rst_{\infty}$ [resp. $\lst_{\infty}$] that if $\bfu \equiv_\mathsf{r} \bfv$ [resp. $\bfu \equiv_\mathsf{l} \bfv$], then $\ev(\bfu) = \ev(\bfv)$, and hence it makes sense to define the evaluation of each element of $\rst_{\infty}$ [resp. $\lst_{\infty}$] to be the evaluation
of any word representing it. Note that $\ev(\bfu) = \ev(\bfv)$ implies that $\con(\bfu) = \con(\bfv)$. By the definitions of the left-stalactic and right-stalactic congruences, the following result holds.

\begin{lemma}[{\cite[Section 2.3]{AR24}}]\label{ipfp}
Let $\bfu, \bfv \in\mathcal{A}_n^{\star}$. Then  $\bfu \equiv_\mathsf{r} \bfv$ if and only if  $\ev(\bfu)=\ev(\bfv)$ and $\fp(\bfu)=\fp(\bfv)$; and $\bfu \equiv_\mathsf{l}\bfv$ if and only if  $\ev(\bfu)=\ev(\bfv)$ and $\ip(\bfu)=\ip(\bfv)$.
\end{lemma}

Note that neither the left-stalactic monoid nor the right-stalactic monoid is anti-automorphic. Then neither the left-stalactic monoid nor the right-stalactic monoid can be equipped with an involution.
In fact, the left-stalactic and right-stalactic monoids of rank $n$ are anti-isomorphic.

\begin{proposition}
The monoids  $\rst_n$ and  $\lst_n$  are anti-isomorphic for each finite $n$.
\end{proposition}
\begin{proof}
Let $\sigma$ be any permutation on $\mathcal{A}_n$.
Then for any $\bfw=a_1a_2\cdots a_k \in \mathcal{A}_n^{\star}$, define a map $\overline{\sigma}$ from $\rst_{n}$ to  $\lst_n$ given by $\overline{\sigma}(\bfw)=\sigma(a_k)\sigma(a_{k-1})\cdots \sigma(a_2)\sigma(a_1)$.
For any $\bfu, \bfv \in\mathcal{A}_n^{\star}$, if $\bfu \equiv_\mathsf{r}\bfv$,  then $\ev(\bfu)=\ev(\bfv)$ and $\fp(\bfu)=\fp(\bfv)$ by Lemma \ref{ipfp}. Without loss of generality, we assume that $\bfu=\bfu_1b_1\bfu_2b_2\cdots\bfu_kb_k$ such that $\bfu_i\in \mathcal{A}_n^{\star}$ and $\fp(\bfu)=b_1b_2\cdots b_k$. It follows from $\fp(\bfu)=\fp(\bfv)$  that $\fp(\bfv)=b_1b_2\cdots b_k$. Then by $\ev(\bfu)=\ev(\bfv)$,  $\bfv$ can be written into
$\bfv =\bfv_1b_1\bfv_2b_2\cdots\bfv_kb_k$ satisfying $\ev(\bfv_1\bfv_2\cdots\bfv_k)=\ev(\bfu_1\bfu_2\cdots\bfu_k)$. By the definition of $\overline{\sigma}$,
\begin{align*}
\overline{\sigma}(\bfu)&=\sigma(b_k)\overline{\sigma}(\bfu_k)\sigma(b_{k-1})\overline{\sigma}(\bfu_{k-1})\cdots \sigma(b_1)\overline{\sigma}(\bfu_1), \\ \overline{\sigma}(\bfv)&=\sigma(b_k)\overline{\sigma}(\bfv_k)\sigma(b_{k-1})\overline{\sigma}(\bfv_{k-1})\cdots \sigma(b_1)\overline{\sigma}(\bfv_1),
\end{align*}
and so $\ev(\overline{\sigma}(\bfu))=\ev(\overline{\sigma}(\bfv))$ and $\ip(\overline{\sigma}(\bfu))=\sigma(b_k)\sigma(b_{k-1})\cdots \sigma(b_1)=\ip(\overline{\sigma}(\bfv))$. Hence  $\overline{\sigma}(\bfu)\equiv_\mathsf{l}\overline{\sigma}(\bfv)$ by Lemma \ref{ipfp}. Symmetrically, $\overline{\sigma}(\bfu)\equiv_\mathsf{l}\overline{\sigma}(\bfv)$ implies $\bfu \equiv_\mathsf{r}\bfv$. Thus $\overline{\sigma}$ is well-defined and injective. Let $\sigma^{-1}$ be the reverse of $\sigma$.
For any $\bfw=a_1a_2\cdots a_k \in \mathcal{A}_n^{\star}$, it is routine to verify that $\overline{\sigma}(\sigma^{-1}(a_k)\sigma^{-1}(a_{k-1})\cdots\sigma^{-1}(a_2)\sigma^{-1}(a_1))=\bfw$. Thus $\overline{\sigma}$ is surjective, and so $\overline{\sigma}$ is bijective.

For any $\bfw=a_1a_2\cdots a_k, \bfw'=a'_1a'_2\cdots a'_{\ell} \in \mathcal{A}_n^{\star}$,
\begin{align*}
\overline{\sigma}(\bfw\bfw')&=\overline{\sigma}(a_1a_2\cdots a_ka'_1a'_2\cdots a'_{\ell})\\
&=\sigma(a'_{\ell})\sigma(a'_{\ell-1})\cdots\sigma(a'_1)\sigma(a_k)\sigma(a_{k-1})\cdots\sigma(a_1)=\overline{\sigma}(\bfw')\overline{\sigma}(\bfw).
\end{align*}
Therefore $\overline{\sigma}$ is an anti-isomorphism  from $\rst_{n}$ to $\lst_{n}$.
\end{proof}

\subsection{The meet-stalactic monoid and its involutions}
For any $\bfu \in \mathcal{A}^{\star}$, define ${\rm P}_{\s}(\bfu)$ to be the  stalactic tableau pair $({\rm P}^\rightarrow_\mathsf{lSt}(\bfu), {\rm P}^\leftarrow_\mathsf{rSt}(\bfu))$. Define the relation
\[
\bfu \equiv_\mathsf{m} \bfv \quad\text{if and only if}\quad  {\rm P}_{\s}(\bfu) =  {\rm P}_{\s}(\bfv)
\]
for all $\bfu, \bfv \in \mathcal{A}^{\star}$.  It is easy to see that
$\equiv_\mathsf{m} = \equiv_\mathsf{r} \wedge \equiv_{\mathsf{l}}$. Thus the relation $\equiv_\mathsf{m}$ is a congruence which is called the \textit{meet-stalactic congruence}.
Define the infinite-rank \textit{meet-stalactic monoid} $\s_{\infty}$ to be the factor monoid $\mathcal{A}^{\star}/\equiv_\mathsf{m}$ and the  \textit{meet-stalactic monoid} of rank $n$, denoted by $\sn$, to be the factor monoid $\mathcal{A}_n^{\star}/\equiv_\mathsf{m}$.
Now by the definition of the meet-stalactic congruence and Lemma \ref{ipfp}, the following result holds.
\begin{proposition}\label{pro:Sij}
Let $\bfu, \bfv \in\mathcal{A}_n^{\star}$. Then  $\bfu \equiv_\mathsf{m} \bfv$ if and only if  $\ev(\bfu)=\ev(\bfv)$, $\ip(\bfu)=\ip(\bfv)$ and $\fp(\bfu)=\fp(\bfv)$.
\end{proposition}

By  \cite[Proposition~3.1]{AR24}, the monoid $\s_{\infty}$ can also be presented by $\langle \mathcal{A}| \mathcal{R}_{\s_{\infty}}\rangle $ where
\begin{equation}
\begin{aligned}
\mathcal{R}_{\s_{\infty}} =&\{ (a\bfu ab\bfv a, a\bfu ba\bfv a): a,b \in \mathcal{A}_n, \bfu,\bfv\in \mathcal{A}_n^{\star}\} \\
\cup & \{(a\bfu ab\bfv b, a\bfu ba\bfv b): a,b \in \mathcal{A}_n, \bfu,\bfv\in \mathcal{A}_n^{\star}\}.
\end{aligned}
\end{equation}

For each $n \in \mathbb{N}$, a presentation for the monoid $\sn$ can be obtained by restricting generators and relations of the above presentation to the generators in $\mathcal{A}_n$.

\begin{proposition}\label{pro:inv}
All involutions of $\sn$ are determined by all elements of order 1 and 2 in the symmetric group on $\mathcal{A}_n$.
\end{proposition}
\begin{proof}
Suppose that $^*$ is an involution of  $\s_n$. Note that in relations (2.1), the evaluation of words are preserved. Then
 for any $k\in \mathcal{A}_n$, there exists some $\ell \in \mathcal{A}_n$ such that $k^*=\ell$ and $\ell^*=k$.
Thus $^*$ is determined by a permutation $\sigma$ on $\mathcal{A}_n$ satisfying $\sigma^2=\sigma_0$ where $\sigma_0$ is the identity permutation on $\mathcal{A}_n$.

Let $\sigma$ be any permutation on $\mathcal{A}_n$ satisfying $\sigma^2=\sigma_0$. Then for any $\bfw=a_1a_2\cdots a_k \in \mathcal{A}_n^{\star}$, define a map from $\s_{n}$ to itself given by $\overline{\sigma}(\bfw)=\sigma(a_k)\sigma(a_{k-1})\cdots \sigma(a_2)\sigma(a_1)$.
For any $\bfu, \bfv \in\mathcal{A}_n^{\star}$, if $\bfu \equiv_\mathsf{m}\bfv$,  then $\ev(\bfu)=\ev(\bfv)$, $\ip(\bfu)=\ip(\bfv)$ and $\fp(\bfu)=\fp(\bfv)$ by Lemma \ref{pro:Sij}. Without loss of generality, we may assume $\bfu=b_0\bfu_1b_1\bfu_2b_2\cdots\bfu_kb_k$ such that $\bfu_i\in \mathcal{A}_n^{\star}$ and $b_0, b_1, b_2, \dots, b_k$ are the first or the last occurrences of all variables in $\bfu$. It follows from $\ip(\bfv)=\ip(\bfu)$, $\fp(\bfv)=\fp(\bfu)$
that $b_0, b_1, b_2, \dots, b_k$ are the first or the last occurrences of all variables in $\bfv$. Then by $\ev(\bfv)=\ev(\bfu)$, $\bfv$ can be written into
$\bfv=b_0\bfv_1b_1\bfv_2b_2\cdots\bfv_kb_k$ satisfying  $\ev(\bfv_1\bfv_2\cdots\bfv_k)=\ev(\bfu_1\bfu_2\cdots\bfu_k)$. By the definition of $\overline{\sigma}$,
\begin{align*}
\overline{\sigma}(\bfu)&=\sigma(b_k)\overline{\sigma}(\bfu_k)\sigma(b_{k-1})\overline{\sigma}(\bfu_{k-1})\cdots \sigma(b_1)\overline{\sigma}(\bfu_1)\sigma(b_0),\\  \overline{\sigma}(\bfv)&=\sigma(b_k)\overline{\sigma}(\bfv_k)\sigma(b_{k-1})\overline{\sigma}(\bfv_{k-1})\cdots \sigma(b_1)\overline{\sigma}(\bfv_1)\sigma(b_0),
\end{align*}
and so $\ev(\overline{\sigma}(\bfu))=\ev(\overline{\sigma}(\bfv))$, $\ip(\overline{\sigma}(\bfu))=\ip(\overline{\sigma}(\bfv))$ and $\fp(\overline{\sigma}(\bfu))=\fp(\overline{\sigma}(\bfv))$. Hence $\overline{\sigma}(\bfu)\equiv_\mathsf{m}\overline{\sigma}(\bfv)$ by Lemma \ref{pro:Sij}. Symmetrically, $\overline{\sigma}(\bfu)\equiv_\mathsf{m}\overline{\sigma}(\bfv)$ implies $\bfu \equiv_\mathsf{m}\bfv$. Thus $\overline{\sigma}$ is well-defined and injective. Since $\sigma^2=\sigma_0$, $\sigma^{-1}=\sigma$.
For any $\bfw=a_1a_2\cdots a_k \in \mathcal{A}_n^{\star}$, it is routine to verify that $\overline{\sigma}(\sigma(a_k)\sigma(a_{k-1})\cdots\sigma(a_2)\sigma(a_1))=\bfw$. Thus $\overline{\sigma}$ is surjective, and so $\overline{\sigma}$ is bijective.

For any $\bfw=a_1a_2\cdots a_k, \bfw'=a'_1a'_2\cdots a'_{\ell} \in \mathcal{A}_n^{\star}$,
\begin{align*}
\overline{\sigma}(\bfw\bfw')&=\overline{\sigma}(a_1a_2\cdots a_ka'_1a'_2\cdots a'_{\ell})\\
&=\sigma(a'_{\ell})\sigma(a'_{\ell-1})\cdots\sigma(a'_1)\sigma(a_k)\sigma(a_{k-1})\cdots\sigma(a_1)=\overline{\sigma}(\bfw')\overline{\sigma}(\bfw).
\end{align*}
Therefore $\overline{\sigma}$ is an anti-automorphism  of $\s_{n}$, that is, $\overline{\sigma}$ is an involution of $\s_{n}$.
\end{proof}

By Proposition 2.5, each element of order 1 and 2 in the symmetric group induces an involution on $\sn$.
Note that each element of order 2 in the symmetric group on $\mathcal{A}_n$ can be decomposed into a product of disjoint 2-cycles.
Clearly, there are at most $\lfloor\frac{n}{2}\rfloor$ disjoint 2-cycles in a permutation on $\mathcal{A}_n$.
Hence by the number of disjoint 2-cycles, the involutions of $\sn$ are divided into $\lfloor\frac{n}{2}\rfloor+1$ types.

{\bf Type 1.} the involution $^{\sigma_0}$ determined by the identity permutation $\sigma_0$ on $\mathcal{A}_n$, that is,
$\sigma_0(1)=1,  \sigma_0(2)=2, \dots, \sigma_0(n)=n$.
It is worth mentioning that $^{\sigma_0}$ is not an identity map on $\s_n$. For example, $([12]_\mathsf{m})^{\sigma_0}= ([2]_\mathsf{m})^{\sigma_0}([1]_\mathsf{m})^{\sigma_0}=[21]_\mathsf{m}\neq [12]_\mathsf{m}$.

{\bf Type 2.} the involution $^{\sigma_1}$ determined by the permutation $\sigma_1$ on $\mathcal{A}_n$ that can be decomposed into one $2$-cycle.

Since there are $\binom{n}{2}$ permutations on $\mathcal{A}_n$ which can be decomposed into one $2$-cycle, these permutations induces  $\binom{n}{2}$ involutions on $\s_n$. In fact, the involution semigroups formed by $\sn$ under each of these $\binom{n}{2}$ involutions are isomorphic.
Let $\tau, \tau'$ be any two permutations which can be decomposed into one $2$-cycle. Suppose that $\tau(i)=j,\tau(j)=i$ and $\tau$ fixes other elements in $\mathcal{A}_n$, and $\tau'(i')=j',\tau'(j')=i'$ and $\tau'$ fixes other elements in $\mathcal{A}_n$. Then $\tau$ and $\tau'$ induce an involution on $\s_n$  respectively, denoted as $^{\tau}$ and $^{\tau'}$. Let $\varphi$ be a map which maps $i$ to $i'$, $j$ to $j'$ and maps elements in $\{1,2, \dots, n\} \setminus\{i, j\}$
to elements in $\{1,2, \dots, n\}\setminus \{i', j'\}$ one by one in order. Then it is routine to verify that $\varphi$ induces an isomorphism between $(\s_n, ^\tau)$ and  $(\s_n, ^{\tau'})$.

{\bf Type 3.} the involution $^{\sigma_2}$ determined by the permutation $\sigma_2$ on $\mathcal{A}_n$ that can be decomposed into two disjoint $2$-cycles.

Clearly $n\geq4$.
Since there are $\binom{n}{2}\cdot\binom{n-2}{2}$
permutations on $\mathcal{A}_n$ which can be decomposed into two $2$-cycles, these permutations induces
$\binom{n}{2}\cdot\binom{n-2}{2}$ involutions on $\s_n$. In fact, the involution semigroups formed by $\sn$ under each of these $\binom{n}{2}\cdot\binom{n-2}{2}$ involutions are isomorphic.
Let $\delta, \delta'$ be any two permutations which can be decomposed into two $2$-cycles. Suppose that $\delta(i)=j,\delta(j)=i, \delta(h)=k,\delta(k)=h$ and $\delta$ fixes other elements in $\mathcal{A}_n$, and $\delta'(i)=j,\delta'(j)=i, \delta'(h)=k,\delta'(k)=h$ and $\delta'$ fixes other elements in $\mathcal{A}_n$. Then $\delta$ and $\delta'$ induce an involution on $\s_n$  respectively, denoted as $^{\delta}$ and $^{\delta'}$. Let $\varphi$ be a map which maps $i$ to $i'$, $j$ to $j'$, $h$ to $h'$, $k$ to $k'$, and maps elements in $\{1,2, \dots, n\} \setminus\{i, j, h, k\}$
to elements in $\{1,2, \dots, n\}\setminus \{i', j', h', k'\}$ one by one in order. Then it is routine to verify that $\varphi$ induces an isomorphism between $(\s_n, ^\delta)$ and  $(\s_n, ^{\delta'})$.

$\cdots\cdots$

{\bf Type $\mathbf{\lfloor\frac{n}{2}\rfloor+1}$.} the involution $^{\sigma_{\lfloor\frac{n}{2}\rfloor}}$ determined by the permutation $\sigma_{\lfloor\frac{n}{2}\rfloor}$ on $\mathcal{A}_n$ that can be decomposed into $\lfloor\frac{n}{2}\rfloor$ disjoint $2$-cycles.

If $n$ is even, then there are $\binom{n}{2}\cdot\binom{n-2}{2}\cdots \binom{4}{2}$
permutations on $\mathcal{A}_n$ which can be decomposed into $\lfloor\frac{n}{2}\rfloor$ disjoint $2$-cycles, these permutations induces
$\binom{n}{2}\cdot\binom{n-2}{2}\cdots \binom{4}{2}$ involutions on $\s_n$. In fact, the involution semigroups formed by $\sn$ under each of these $\binom{n}{2}\cdot\binom{n-2}{2}\cdots \binom{4}{2}$ involutions are isomorphic.
Let $\eta, \eta'$ be any two permutations which can be decomposed into $\lfloor\frac{n}{2}\rfloor$ disjoint $2$-cycles. Suppose that
\begin{align*}
\eta(i_1)=j_1,\eta(j_1)=i_1, \eta(i_2)=j_2,\eta(j_2)=i_2, \dots, \eta(i_{\lfloor\frac{n}{2}\rfloor})=j_{\lfloor\frac{n}{2}\rfloor},\eta(j_{\lfloor\frac{n}{2}\rfloor})=i_{\lfloor\frac{n}{2}\rfloor},\\
\eta(i'_1)=j'_1,\eta(j'_1)=i'_1, \eta(i'_2)=j'_2,\eta(j'_2)=i'_2, \dots, \eta(i'_{\lfloor\frac{n}{2}\rfloor})=j'_{\lfloor\frac{n}{2}\rfloor},\eta(j'_{\lfloor\frac{n}{2}\rfloor})=i'_{\lfloor\frac{n}{2}\rfloor}.
\end{align*}
\noindent Then $\eta$ and $\eta'$ induce an involution on $\s_n$ respectively, denoted as $^{\eta}$ and $^{\eta'}$. Let $\varphi$ be a map which maps $i_{\ell}$ to $i'_{\ell}$, $j_{\ell}$ to $j'_{\ell}$ for $1\leq \ell \leq \lfloor\frac{n}{2}\rfloor$. Then it is routine to verify that $\varphi$ induces an isomorphism between $(\s_n, ^\eta)$ and  $(\s_n, ^{\eta'})$.

If $n$ is odd, then there are $\binom{n}{2}\cdot\binom{n-2}{2}\cdots \binom{3}{2}$
permutations on $\mathcal{A}_n$ which can be decomposed into $\lfloor\frac{n}{2}\rfloor$ disjoint $2$-cycles, these permutations induces
$\binom{n}{2}\cdot\binom{n-2}{2}\cdots \binom{3}{2}$ involutions on $\s_n$. Similar to the case when $n$ is even, the involution semigroups formed by $\sn$ under each of these $\binom{n}{2}\cdot\binom{n-2}{2}\cdots \binom{3}{2}$ involutions are isomorphic.

\section{Representations of $\s_n$ with involution}\label{}%
In \cite[Theorem 4.4]{CJKM21}, Cain et al. have given a faithful representation of the right-stalactic monoid. According to this result, a faithful representation of $\s_n$ under each involution is obtained.

\begin{theorem}[{\cite[Theorem 4.4]{CJKM21}}]
Let $\mathbb{S}$ be a commutative unital semiring with zero containing an element
of infinite multiplicative order. The right-stalactic monoid of rank $n$ admits a faithful representation by upper triangular matrices of size $n^2$ over $\mathbb{S}$ having block-diagonal structure with largest block of size $2$ (or size $1$ if $n = 1$).
\end{theorem}

Denote by $\psi_r$ the representation of right-stalactic monoid given in \cite[Theorems 4.3 and 4.4]{CJKM21}. Let $\sigma$ be any permutation on $\mathcal{A}_n$. For any $\bfw=a_1a_2\cdots a_k \in \mathcal{A}_n^{\star}$, define a map $\overline{\sigma}$ from $\lst_n$ to $\rst_n$ given by $\overline{\sigma}(\bfw)=\sigma(a_k)\sigma(a_{k-1})\cdots \sigma(a_2)\sigma(a_1)$.  By Proposition 2.3, $\overline{\sigma}$ is an anti-isomorphism between $\rst_n$ and $\lst_n$. Define a map $\psi_{\ell}: \mathsf{lSt}_n\rightarrow UT_{n^2}(\mathbb{S})$ given by $ \bfw\mapsto (\psi_r(\overline{\sigma}(\bfw)))^{D}$.
\begin{theorem}
The map $\psi_{\ell}:\mathsf{lSt}_n \rightarrow UT_{n^2}(\mathbb{S})$ is a faithful representation of $\mathsf{lSt}_n$.
\end{theorem}
\begin{proof}
Clearly, $\psi_{\ell}$ is injective by its definition. It suffices to show that for any $\bfw_1, \bfw_2\in \mathcal{A}_n^{\star}$,
$\psi_{\ell}(\bfw_1\bfw_2)=\psi_{\ell}(\bfw_1)\psi_{\ell}(\bfw_2)$.
Denote by $\bfw_1=b_1b_2\cdots b_m$ and  $\bfw_2=c_1c_2\cdots c_n$. Then
\begin{align*}
&\psi_{\ell}(\bfw_1\bfw_2)\\
&=\psi_{\ell}(b_1b_2\cdots b_mc_1c_2\cdots c_n)\\
&=(\psi_r(\overline{\sigma}(b_1b_2\cdots b_mc_1c_2\cdots c_n)))^{D}\\
&=(\psi_r(\sigma(c_n)\sigma(c_{n-1})\cdots\sigma(c_1)\sigma(b_m)\sigma(b_{m-1})\cdots\sigma(b_1)))^D\\
&=(\psi_r(\sigma(c_n))\psi_r(\sigma(c_{n-1}))\cdots\psi_r(\sigma(c_1))\psi_r(\sigma(b_m))\psi_r(\sigma(b_{m-1}))\cdots\psi_r(\sigma(b_1)))^D\\
&=(\psi_r(\sigma(b_1)))^D(\psi_r(\sigma(b_{2})))^D\cdots(\psi_r(\sigma(b_m)))^D(\psi_r(\sigma(c_1)))^D(\psi_r(\sigma(c_{2})))^D\cdots(\psi_r(\sigma(c_n)))^D\\
&=(\psi_r(\sigma(b_m))\psi_r(\sigma(b_{m-1}))\cdots\psi_r(\sigma(b_1)))^D(\psi_r(\sigma(c_n))\psi_r(\sigma(c_{n-1}))\cdots\psi_r(\sigma(c_1)))^D\\
&=(\psi_r(\sigma(b_m)\sigma(b_{m-1})\cdots\sigma(b_1)))^D(\psi_r(\sigma(c_n)\sigma(c_{n-1})\cdots\sigma(c_1)))^D\\
&=(\psi_r(\overline{\sigma}(\bfw_1)))^D(\psi_r(\overline{\sigma}(\bfw_2)))^D\\
&=\psi_{\ell}(\bfw_1)\psi_{\ell}(\bfw_2),
\end{align*}
as required.
\end{proof}

Recall that for $i=0,1,2, \dots, \lfloor\frac{n}{2}\rfloor$, $\sigma_i $ is the permutation which can be decomposed into $i$ disjoint 2-cycles and $\sigma_i $ induces an involution $^{\sigma_i}$ on $\sn$.
For any $\bfw=a_1a_2\cdots a_k \in \mathcal{A}_n^{\star}$, define a map $\overline{\sigma_i}$ from $\lst_n$ to $\rst_n$ given by $\overline{\sigma_i}(\bfw)=\sigma_i(a_k)\sigma_i(a_{k-1})\cdots \sigma_i(a_2)\sigma_i(a_1)$.
By Proposition 2.3, $\overline{\sigma_i} $ is an anti-isomorphism between $\rst_n$ and $\lst_n$. Define $\psi_{\ell}^{(i)}: \mathsf{lSt}_n\rightarrow UT_{n^2}(\mathbb{S})$ given by $ \bfw\mapsto (\psi_r(\overline{\sigma_i}(\bfw)))^{D}$.
By Theorem 3.2, $\psi_{\ell}^{(i)}:\mathsf{lSt}_n \rightarrow UT_{n^2}(\mathbb{S})$ is a faithful representation of $\mathsf{lSt}_n$.
 Denote by $\mathsf{diag}\{\Lambda_1,\Lambda_2,\dots, \Lambda_n\}$ the block diagonal matrix
\[
\begin{pmatrix}
\Lambda_1 & & & \\
 & \Lambda_2 & &\\
 & & \ddots &\\
  &&& \Lambda_n
\end{pmatrix}
\]
where $\Lambda_1,\Lambda_2,\dots, \Lambda_n$ are square matrices.
Define a map $\psi^{(i)}: \mathsf{mSt}_n\rightarrow UT_{2n^2}(\mathbb{S})$ given by $ \bfw\mapsto \mathsf{diag}\{\psi_{\ell}^{(i)}(\bfw),\psi_r(\bfw)\}$.

\begin{theorem}
The map $\psi^{(i)}:(\mathsf{mSt}_n,^{\sigma_i}) \rightarrow (UT_{2n^2}, ^D)$ is a faithful representation of $(\mathsf{mSt}_n,^{\sigma_i})$ for $i=0,1,2, \dots, \lfloor\frac{n}{2}\rfloor$.
\end{theorem}

\begin{proof}
By the definition of $\s_n$ and Theorems 3.1 and 3.2, $\psi^{(i)}$ is a faithful representation of $\mathsf{mSt}_n$. To show that $\psi^{(i)}$ is a faithful representation of $(\mathsf{mSt}_n,^{\sigma_i})$, we only need to show that for any $\bfw\in \mathcal{A}_n^{\star}$, $\psi^{(i)}(\bfw^{\sigma_i})=(\psi^{(i)}(\bfw))^D$. It suffices to show that for any $j\in \mathcal{A}_n$, $\psi^{(i)}(j^{\sigma_i})=(\psi^{(i)}(j))^D$.
By the definition of $\psi^{(i)}$,
\begin{align*}
\psi^{(i)}(j^{\sigma_i})&=\psi^{(i)}(\sigma_i(j))\\
&= \mathsf{diag}\{\psi_{\ell}^{(i)}(\sigma_i(j)),\psi_r(\sigma_i(j))\}\\
&= \mathsf{diag}\{(\psi_r(\sigma_i(\sigma_i(j))))^D,\psi_r(\sigma_i(j))\}\\
&= \mathsf{diag}\{(\psi_r(j))^D,\psi_r(\sigma_i(j))\}\\
&= \mathsf{diag}\{(\psi_r(\sigma_i(j)))^D,\psi_r(j)\}^D\\
&= (\mathsf{diag}\{\psi_{\ell}^{(i)}(j),\psi_r(j)\})^D\\
&=(\psi^{(i)}(j))^D,
\end{align*}
as required.
\end{proof}

\section{Identities of $(\sn, ^{\sigma_0})$}\label{sec:3.2}%

In this section, we investigate the identities of $\sn$ with involution $^{\sigma_0}$ determined by the identity permutation on $\mathcal{A}_n$.
We prove that for all finite $n\geq 2$, $(\sn, ^{\sigma_0})$ satisfy the same identities. Further,
we give characterizations of the identities satisfied by them; based on these results,
the finite basis problem and identity checking problem for $(\sn, ^{\sigma_0})$ are solved.

\subsection{Embedding into a direct product of copies of $(\s_2, ^{\sigma_0})$}
For $n >  2$, we show that $(\sn,^{\sigma_0})$  can be embedded into a direct product of copies of $(\s_2,^{\sigma_0})$.
For any $i,j \in \mathcal{A}_n$ with $i < j$, define a map $\varphi_{ij}$ from $\mathcal{A}_n$ to $\s_2$ in the following way: for any $k \in \mathcal{A}_n$,
\begin{align*}
k &\longmapsto \begin{cases}
[1]_{\m} & \text{if } k = i,\\
[2]_{\m} & \text{if } k = j,\\
\left[\varepsilon\right]_{\m} & \text{otherwise}.
\end{cases}
\end{align*}
It is routine to verify that the map $\varphi_{ij}$ can be extended to a homomorphism from $(\mathcal{A}^{\star}_n, ^{\sigma_0})$ to $(\s_2, ^{\sigma_0})$.
In fact, the following result holds.



\begin{lemma}
\label{lem:varphi_equiv_sn}
Let $\bfu, \bfv \in \mathcal{A}_n^{\star}$. Then $\bfu \equiv_{\mathsf{m}} \bfv$ if and only if $\varphi_{ij}(\bfu) \equiv_{\mathsf{m}} \varphi_{ij}(\bfv)$ for all $1 \leq i < j \leq n$.
\end{lemma}
	
\begin{proof}
To prove the necessity, it suffices to show that for any $i,j$ with $1 \leq i < j \leq n$, if $\bfu\equiv_{\mathsf{m}}\bfv$, then $\varphi_{ij}(\bfu)\equiv_{\mathsf{m}}\varphi_{ij}(\bfv)$.
If $\bfu\equiv_{\mathsf{m}}\bfv$, then $\ev(\bfu)=\ev(\bfv), \ip(\bfu)=\ip(\bfv)$ and $\fp(\bfu)=\fp(\bfv)$ by Proposition \ref{pro:Sij}.  In particular, the number of occurrences of $i$ and $j$ in $\bfu$ is the same as that of $\bfv$, respectively; and the order of $i$ and $j$ occurring  in $\bfu$ is the same as that of $\bfv$. Note that
the occurrences of $i$ in $\bfu$ and $\bfv$ correspond to the  the occurrences of $1$ in $\varphi_{ij}(\bfu)$ and $\varphi_{ij}(\bfv)$, respectively;
and the occurrences of $j$ in $\bfu$ and $\bfv$ correspond to the  the occurrences of $2$ in $\varphi_{ij}(\bfu)$ and $\varphi_{ij}(\bfv)$, respectively.
Therefore  $\ev(\varphi_{ij}(\bfu))=\ev(\varphi_{ij}(\bfv))$, $\ip(\varphi_{ij}(\bfu))=\ip(\varphi_{ij}(\bfv))$ and $\fp(\varphi_{ij}(\bfu))=\fp(\varphi_{ij}(\bfu))$, that is $\varphi_{ij}(\bfu)\equiv_{\mathsf{m}}\varphi_{ij}(\bfv)$ by Proposition \ref{pro:Sij}.
The sufficiency can be obtained symmetrically.
\end{proof}


For each $n \geq 3$, let $I_n$ be the index set
\[
\left\{ (i,j): 1 \leq i < j \leq n \right\}.
\]
Now, consider the map
\[
\phi_n : (\sn, ^{\sigma_0}) \longrightarrow \prod\limits_{I_n} (\s_2, ^{\sigma_0}),
\]
whose $(i,j)$-th component is given by $\varphi_{ij}([\bfw]_{\mathsf{m}})$ for $\bfw \in \mathcal{A}_n^{\star}$ and $(i,j) \in I_n$.
		
\begin{lemma}
\label{lem:staln_embed_stal2}
The map $\phi_n$ is an embedding.
\end{lemma}
	
\begin{proof}
It follows from the definition of $\phi_n$ and Lemma~\ref{lem:varphi_equiv_sn} that $\phi_n$ is a homomorphism and for any $\bfu, \bfv \in \mathcal{A}_n^{\star}$, $\bfu \equiv_{\mathsf{m}} \bfv$ if and only if $\phi_n([\bfu]_{\mathsf{m}}) = \phi_n([\bfv]_{\mathsf{m}})$. Therefore $\phi_n$ is an embedding.
\end{proof}

\begin{theorem}\label{thm:stal-id-stal2}
For any $n \geq 3$, $(\s_n, ^{\sigma_0})$ and $(\s_2, ^{\sigma_0})$ satisfy exactly the same identities.
\end{theorem}
\begin{proof}
By Lemma \ref{lem:staln_embed_stal2}, for any $n \geq 3$, $(\sn, ^{\sigma_0})$ can be embedded into a direct product of $\binom{n}{2}$ copies of $(\s_2, ^{\sigma_0})$.
Thus  $(\sn, ^{\sigma_0})$ is in the variety generated by $(\s_2, ^{\sigma_0})$. Since $(\s_2, ^{\sigma_0})$ is a submonoid of $(\sn, ^{\sigma_0})$ for any $n \geq 3$, $(\s_2, ^{\sigma_0})$ is in the variety generated by $(\sn, ^{\sigma_0})$.
Therefore they satisfy exactly the same identities.
\end{proof}

\subsection{Characterizations of identities and identity basis}
For convenience, denote by $_{i}x$ the $i$-th from the left occurrence of $x$ in $\bfu$, $_{\infty}x$ the last occurrence of $x$ in $\bfu$, and $_{i}x \prec_{\bfu} {_{j}y}$ if and only if the $_{i}x$ precedes the $_{j}y$ in $\bfu$.

\begin{theorem}\label{lem:inv1ids2}
An identity $\bfu\approx \bfv$ holds in $(\s_2,^{\sigma_0})$ if and only if
\begin{enumerate}[\rm(i)]
  \item $\occ(x, \overline{\bfu})= \occ(x, \overline{\bfv})$ for any $x\in \con(\overline{\bfu\bfv})$.
  \item  for any $x,y \in \con(\bfu\bfv)$ with $x, x^*\neq y$, the first [resp. last] variable of $\bfu[x,y]$ is the same as that of $\bfv[x,y]$.
\end{enumerate}
\end{theorem}
\begin{proof}
Let $\bfu\approx \bfv$ be any identity of $(\s_2,^{\sigma_0})$. Suppose that there exists some $x\in \con(\overline{\bfu\bfv})$ such that $\occ(x, \overline{\bfu})\neq \occ(x, \overline{\bfv})$. Let $\theta_1$ be the substitution that maps $x$ to $[1]_\mathsf{m}$ and any other variable to $[\varepsilon]_\mathsf{m}$. Then $\theta_1(\bfu)=[1^{\occ(x, \overline{\bfu})}]_\mathsf{m}\neq [1^{\occ(x, \overline{\bfv})}]_\mathsf{m}=\theta_1(\bfv)$, a contradiction. Therefore the condition (i) holds.

Suppose that there exist $x,y \in \con(\bfu\bfv)$ with $x, x^*\neq y$ such that $\bfu[x,y]$ starts with $x$ but $\bfv[x,y]$ does not. If $\bfv[x,y]$ starts with $x^*$, then let $\theta_2$ be the substitution such that $x\mapsto [12]_\mathsf{m}, y\mapsto [\varepsilon]_\mathsf{m}$, we obtain
\begin{align*}
\theta_2(\bfu[x,y])=[12]_\mathsf{m}\cdot [\bfs]_\mathsf{m}\neq [21]_\mathsf{m}\cdot [\bft]_\mathsf{m}=\theta_2(\bfv[x,y])
\end{align*}
where $\bfs, \bft \in \mathcal{A}_2^{\star}$, a contradiction; if $\bfv[x,y]$ starts with $y$, then let $\theta_3$ be the substitution such that $x\mapsto [1]_\mathsf{m}, y\mapsto [2]_\mathsf{m}$, we obtain
\begin{align*}
\theta_3(\bfu[x,y])=[1]_\mathsf{m}\cdot [\bfs]_\mathsf{m}\neq [2]_\mathsf{m}\cdot [\bft]_\mathsf{m}=\theta_3(\bfv[x,y])
\end{align*}
where $\bfs, \bft \in \mathcal{A}_2^{\star}$, a contradiction. Thus $\bfu[x,y]$ and $\bfv[x,y]$ start with the same variable. Symmetrically, $\bfu[x,y]$ and $\bfv[x,y]$ end with the same variable. Therefore the condition (ii) holds.

Conversely, let $\bfu\approx \bfv$ be any identity satisfying (i) and (ii) and $\phi$ be any substitution from $\mathcal{X}$ to $\s_2$.
It follows from (i) and the definition of the involution $^{\sigma_0}$ that $\ev(\phi(\bfu))=\ev(\phi(\bfv))$. If $\con(\phi(\bfu))=\{1\}$ or $\{2\}$, then $\phi(\bfu)=\phi(\bfv)$ by (i) and the definition of $^{\sigma_0}$.
If $\con(\phi(\bfu))=\{1,2 \}$, to show $\phi(\bfu)=\phi(\bfv)$, it suffices to show that $\ip(\phi(\bfu))=\ip(\phi(\bfv))$ and $\fp(\phi(\bfu))=\fp(\phi(\bfv))$. By symmetry, we only need to show that $\ip(\phi(\bfu))=\ip(\phi(\bfv))$.

If ${_11}\prec_{\phi(\bfu)} {_12}$, assume that $\phi(x)\neq [\varepsilon]_\mathsf{m}$ for any $x\in \con(\bfu)$ by Remark \ref{rem:delete}, then $\bfu$ can be written into the form $\bfu = z\bfu_1$ satisfying $\phi(z)$ starts with $1$.  It follows from (ii) that $\bfv = z\bfv_1$.
Thus ${_11}\prec_{\phi(\bfu)} {_12}$  implies ${_11}\prec_{\phi(\bfv)} {_12}$, and by symmetry we have that ${_11}\prec_{\phi(\bfu)} {_12}$  if and only if ${_11}\prec_{\phi(\bfv)} {_12}$.
 Therefore  $\ip(\phi(\bfu))=\ip(\phi(\bfv))$.
\end{proof}

An immediate consequence of Theorem \ref{lem:inv1ids2} is the following.

\begin{corollary}
The decision problem ${\textsc{Check-Id}}(\sn,^{\sigma_0})$ for each finite $n$
belongs to the complexity class $\mathsf{P}$.
\end{corollary}
\begin{proof}
For ${\textsc{Check-Id}}(\sn,^{\sigma_0})$, given any identity $\bfu\approx \bfv$, it suffices to show that one can check whether the words $\bfu$ and  $\bfv$ satisfy conditions of Theorem \ref{lem:inv1ids2} in polynomial time.  To check whether $\occ(x, \overline{\bfu})= \occ(x, \overline{\bfv})$ holds for any $x\in \con(\overline{\bfu\bfv})$, it suffices to check whether $\con(\overline{\bfu}) =\con(\overline{\bfv})$ and $\occ(x, \overline{\bfu})=  \occ(x, \overline{\bfv})$ for any $x\in \con(\overline{\bfu\bfv})$. To check whether condition (ii) of Theorem \ref{lem:inv1ids2} holds, it suffices to check whether the first [resp. last] variable in $\bfu[x, y]$ is the same as that of $\bfv[x, y]$ for any $x, x^*\neq y \in \con(\bfu\bfv)$. Clearly, this can be done in polynomial time.
\end{proof}

In the end of this section, we prove that $(\s_2,^{\sigma_0})$ is finitely based by giving a finite identity basis.

\begin{theorem}\label{thm:FBS2inv1}
The variety generated by  $(\s_2,^{\sigma_0})$ is defined by the  identities \eqref{id: inv} and
\begin{gather}
x^{\circledast_1}hx^*kx^{\circledast_2}\approx x^{\circledast_1}hxkx^{\circledast_2}, \label{id:xhxtx**}\\
x^{\circledast_1}hxykx^{\circledast_2}\approx x^{\circledast_1}hyxkx^{\circledast_2}, \label{id:xhxtx*} \\
x^{\circledast_1}hyxky^{\circledast_2} \approx x^{\circledast_1}hxyky^{\circledast_2}, \label{id:xhxyty*}
\end{gather}
where $\circledast_1, \circledast_2\in\{1, *\}$.
\end{theorem}
\begin{proof}
Clearly, $(\s_2,^{\sigma_0})$ satisfies the identities \eqref{id:xhxtx**}--\eqref{id:xhxyty*} by Theorem~\ref{lem:inv1ids2}.
It suffices to show that each non-trivial identity satisfied by $(\s_2,^{\sigma_0})$ can be deduced from \eqref{id:xhxtx**}--\eqref{id:xhxyty*}. By the identity \eqref{id:xhxtx**}, any non-balanced identity of $(\s_2,^{\sigma_0})$ can be converted into some balanced identity. It suffices to show that each non-trivial balanced identity satisfied by $(\s_2,^{\sigma_0})$ can be deduced from \eqref{id:xhxtx*} and \eqref{id:xhxyty*}.
Let $\Sigma$ be the set of all non-trivial balanced identities satisfied by $(\s_2,^{\sigma_0})$ but  not  deducible from \eqref{id:xhxtx*} and \eqref{id:xhxyty*}. Suppose that $\Sigma\neq \emptyset$. Note that any balanced identity $\bfu\approx \bfv$ in $\Sigma$ can be written
uniquely in the form $\bfu' a\bfw \approx \bfv' b\bfw$ where $a,b$ are distinct variables and $|\bfu'|=|\bfv'|$. Choose an identity, say $\bfu\approx \bfv$, from $\Sigma$ such that when it is written as $\bfu' a\bfw\approx \bfv' b\bfw$,
the lengths of the words $\bfu'$ and $\bfv'$ are as short as possible. Since $\bfu\approx \bfv$ is balanced, we have $\con(\bfu'a)=\con(\bfv'b)$ and $\occ(x, \bfu'a)=\occ(x, \bfv'b)$ for any $x \in \con(\bfu'a)$.

Since $\con(\bfu'a)=\con(\bfv'b)$, we have $b\in \con(\bfu')$, and so $\bfu'a=\bfu_1 bc\bfu_2$ with $b\neq c$ and $b\not\in\con(\bfu_2)$.  Since $\con(\bfu'a)=\con(\bfv'b)$, we have $c\in \con(\bfv')$, and so $\bfv'=\bfv_1 c\bfv_2$ with $c\not\in\con(\bfv_2)$. Thus $\bfu=\bfu_1 bc\bfu_2\bfw$ and $\bfv=\bfv_1 c\bfv_2b\bfw$.
On the one hand, we claim that $b\in\con(\bfu_1)$ or $b^*\in\con(\bfu_1)$ holds when $b^*=c$ and that at least one of the four cases $b\in\con(\bfu_1)$, $c\in\con(\bfu_1)$, $b^*\in\con(\bfu_1)$ and $c^*\in\con(\bfu_1)$ is true when $b^*\neq c$. First, we consider the case $b^*=c$. If $b, b^*\not\in\con(\bfu_1)$, then $\occ(b, \bfu'a)=1$, and so $\bfu[b]$ starts with $b$ but $\bfv[b]$ does not, which contradicts (ii) of Theorem~\ref{lem:inv1ids2}. Next, we consider the case $b^*\neq c$. If $b, c, b^*, c^* \not\in\con(\bfu_1)$, then $\occ(b, \bfu'a)=1$, and so $\bfu[b,c]$ starts with $b$ but $\bfv[b,c]$ does not, which contradicts (ii) of Theorem~\ref{lem:inv1ids2}.
On the other hand,
we claim that $b\in\con(\bfu_2\bfw)$ or $b^*\in\con(\bfu_2\bfw)$ holds when $b^*=c$ and that at least one of the four cases $b\in\con(\bfu_2\bfw)$, $c\in\con(\bfu_2\bfw)$, $b^*\in\con(\bfu_2\bfw)$ and $c^*\in\con(\bfu_2\bfw)$ is true when $b^*\neq c$. First, we consider the case $b^*=c$. If $b, b^*\not\in\con(\bfu_2\bfw)$, then $\bfu[b]$ ends with $b^*$ but $\bfv[b]$ does not, which contradicts (ii) of Theorem~\ref{lem:inv1ids2}. Next, we consider the case $b^*\neq c$. If $b, c, b^*, c^*\not\in\con(\bfu_2\bfw)$, then $\bfu[b,c]$ ends with $c$ but $\bfv[b,c]$ does not, which contradicts (ii) of Theorem~\ref{lem:inv1ids2}.
Therefore, the
identities \eqref{id:xhxtx*} and \eqref{id:xhxyty*} can be used to convert $\bfu_1 bc\bfu_2 \bfw$ into $\bfu_1cb\bfu_2\bfw$. By repeating this process, the word $\bfu=\bfu'a\bfw=\bfu_1bc\bfu_2 \bfw$ can be converted into the word $\bfu_1c\bfu_2b \bfw$ by the
identities \eqref{id:xhxtx*} and \eqref{id:xhxyty*}.

Clearly the identities $\bfu_1c\bfu_2b\bfw \approx \bfu'a\bfw \approx \bfv'b\bfw$ hold in $(\s_2, ^{\sigma_0})$. Note that $|\bfu_1c\bfu_2b\bfw|=|\bfu'a\bfw|=|\bfv'b\bfw|$ and words in the identity $\bfu_1c\bfu_2b\bfw \approx \bfv'b\bfw$ have a longer common suffix than words in the identity $\bfu'a\bfw \approx \bfv'b\bfw$. Hence $\bfu_1c\bfu_2b\bfw \approx \bfv'b\bfw \notin \Sigma$ by the minimality assumption on the lengths of $\bfu'$ and $\bfv'$, that is, $\bfu_1c\bfu_2b\bfw \approx \bfv'b\bfw$ can be deducible from \eqref{id:xhxtx*} and \eqref{id:xhxyty*}. We have shown that $\bfu'a\bfw \approx \bfu_1c\bfu_2b\bfw$ can be deducible from \eqref{id:xhxtx*} and \eqref{id:xhxyty*}. Hence $\bfu'a\bfw \approx \bfv'b\bfw$ can be deducible from \eqref{id:xhxtx*} and \eqref{id:xhxyty*}, which contradicts $\bfu'a\bfw \approx \bfv'b\bfw \in \Sigma$. Therefore $\Sigma=\emptyset$.
\end{proof}

\section{Identities of $(\sn, ^{\sigma_1})$}\label{sec:4.2}%
In this section, we investigate the identities of $\sn$ with involution determined by the permutation that can be decomposed into one 2-cycle.
Recall that each permutation that can be decomposed into one 2-cycle can induce an involution on $\sn$.
Involution semigroups formed by $\sn$ under each of these involutions are isomorphic. For convenience, in this section, we only consider the involution $^{\sigma_1}$ determined by the following permutation $\sigma_1$:
$$\sigma_1(1)=n,  \sigma_1(n)=1, \text{and}\,\, \sigma_1(2)=2, \sigma_1(3)=3, \dots, \sigma_1(n-1)=n-1.$$
We prove that for all finite $n\geq 3$, $(\sn, ^{\sigma_1})$ satisfy the same identities. Further,
we give characterizations of the identities satisfied by $(\s_2, ^{\sigma_1})$ and $(\s_3, ^{\sigma_1})$; based on these results,
the finite basis problem and identity checking problem for them are solved.

\subsection{Embedding into a direct product of copies of $(\s_3, ^{\sigma_1})$}
For $n >  3$, we show that $(\sn,^{\sigma_1})$ can be embedded into a direct product of copies of $(\s_3,^{\sigma_1})$.
For any $i<j \in \mathcal{A}_n$ with $n>3$, it follows from the definition of $^{\sigma_1}$ that there are four cases about the order of $i,j, i^{\sigma_1}, j^{\sigma_1}$ in $\mathcal{A}_n$: $i^{\sigma_1}=j$,  $i<j=j^{\sigma_1}<i^{\sigma_1}$, $j^{\sigma_1}<i=i^{\sigma_1}<j$, $i^{\sigma_1}=i< j=j^{\sigma_1}$.
For any $i<j \in \mathcal{A}_n$ with $n>3$, define a map $\varphi_{ij}$ from $\mathcal{A}_n^{\star}$ to $\s_3$ which can be determined by the following three cases according to the order of $i, i^{\sigma_1}, j, j^{\sigma_1}$ in $\mathcal{A}_n$.

{\bf Case 1.} $i^{\sigma_1} = j$. Define a map $\lambda_{ij}: \mathcal{A}_n \rightarrow \s_3$ given by
\begin{align*}
k &\mapsto \begin{cases}
[1]_{\m}, & \text{if}\ k = i,\\
[3]_{\m}, & \text{if}\ k = j,\\
\left[\varepsilon\right]_{\m}, & \text{otherwise}.
\end{cases}
\end{align*}
Clearly, this map can be extended to a homomorphism $\lambda_{ij}: \mathcal{A}_n^{\star} \rightarrow \s_3$.
Further, $\lambda_{ij}$
is also a homomorphism from $(\mathcal{A}_n^{\star},^{\sigma_1})$ to $(\s_3,^{\sigma_1})$. This is because for any $k \in \mathcal{A}_n$, $\lambda_{ij}(k^{\sigma_1})=(\lambda_{ij}(k))^{\sigma_1}$ by
\begin{align*}
\begin{cases}
\lambda_{ij}(k^{\sigma_1})=[3]_{\m}=(\lambda_{ij}(k))^{\sigma_1}, & \text{if}\ k = i,\\
\lambda_{ij}(k^{\sigma_1})=[1]_{\m}=(\lambda_{ij}(k))^{\sigma_1},  & \text{if}\ k = j,\\
\lambda_{ij}(k^{\sigma_1})=[\varepsilon]_{\m}=(\lambda_{ij}(k))^{\sigma_1},   & \text{otherwise}.
\end{cases}
\end{align*}

{\bf Case 2.} $i<j=j^{\sigma_1}<i^{\sigma_1}$ or $j^{\sigma_1}<i=i^{\sigma_1}<j$. For convenience, let $i_1=i, i_2=j$ and $i_3=i^{\sigma_1}$ when $i<j=j^{\sigma_1}<i^{\sigma_1}$ and $i_1=j^{\sigma_1}, i_2=i$ and $i_3=j$ when $j^{\sigma_1}<i=i^{\sigma_1}<j$. Define a map $\theta_{ij}: \mathcal{A}_n \rightarrow \s_3$ by
\begin{align*}
k \mapsto \begin{cases}
[1]_{\m}, & \text{if}\ k = i_1,\\
[2]_{\m}, & \text{if}\ k = i_2,\\
[3]_{\m}, & \text{if}\ k = i_3,\\
\left[\varepsilon\right]_{\m}, & \text{otherwise.}
\end{cases}
\end{align*}
Clearly, $\theta_{ij}$ can be extended to a homomorphism from $\mathcal{A}_n^{\star}$ to $\s_3$.
Further, $\theta_{ij}$ is also a homomorphism from $(\mathcal{A}_n^{\star},^{\sigma_1})$ to $(\s_3,^{\sigma_1})$. This is because for any $k \in \mathcal{A}_n$, $\theta_{ij}(k^{\sigma_1})=(\theta_{ij}(k))^{\sigma_1}$ by
\begin{align*}
\begin{cases}
\theta_{ij}(k^{\sigma_1})=[3]_{\m}= (\theta_{ij}(k))^{\sigma_1}, & \text{if}\ k = i_1,\\
\theta_{ij}(k^{\sigma_1})=[2]_{\m}=(\theta_{ij}(k))^{\sigma_1},  & \text{if}\ k = i_2,\\
\theta_{ij}(k^{\sigma_1})=[1]_{\m}=(\theta_{ij}(k))^{\sigma_1},  & \text{if}\ k =i_3,\\
\theta_{ij}(k^{\sigma_1})=[\varepsilon]_{\m}=(\theta_{ij}(k))^{\sigma_1},  & \text{otherwise}.
\end{cases}
\end{align*}

{\bf Case 3.} $i^{\sigma_1}=i<j=j^{\sigma_1}$.  Define a map $\eta_{ij}: \mathcal{A}_n \rightarrow \s_3$ by
\begin{align*}
k \mapsto \begin{cases}
[13]_{\m}, & \text{if}\ k = i,\\
[31]_{\m}, & \text{if}\ k = j,\\
\left[\varepsilon\right]_{\m}, & \text{otherwise.}
\end{cases}
\end{align*}
Clearly, $\eta_{ij}$ can be extended to a homomorphism from $\mathcal{A}_n^{\star}$ to $\s_3$.
Further, $\eta_{ij}$ is also a homomorphism from $(\mathcal{A}_n^{\star},^{\sigma_1})$ to $(\s_3,^{\sigma_1})$. This is because for any $k \in \mathcal{A}_n$, $\eta_{ij}(k^{\sigma_1})=(\eta_{ij}(k))^{\sigma_1}$ by
\begin{align*}
\begin{cases}
\eta_{ij}(k^{\sigma_1})=[13]_{\m}= (\eta_{ij}(k))^{\sigma_1}, & \text{if}\ k = i,\\
\eta_{ij}(k^{\sigma_1})=[31]_{\m}=(\eta_{ij}(k))^{\sigma_1},  & \text{if}\ k =j,\\
\eta_{ij}(k^{\sigma_1})=[\varepsilon]_{\m}\,\,\,=(\eta_{ij}(k))^{\sigma_1},  & \text{otherwise}.
\end{cases}
\end{align*}
	
Now we can define the map $\varphi_{ij}: \mathcal{A}^{\star}_n \rightarrow \s_3$ by
\begin{align*}
\varphi_{ij}=
\left\{
\begin{array}{ll}
\lambda_{ij} & \hbox{if $i^{\sigma_1}=j$,} \\
[0.1cm]
\theta_{ij} & \hbox{if $i<j=j^{\sigma_1}<i^{\sigma_1}$ or $j^{\sigma_1}<i=i^{\sigma_1}<j$,} \\
[0.1cm]
\eta_{ij} & \hbox{if $i^{\sigma_1}=i<j=j^{\sigma_1}<i$,}
\end{array}
\right.
\end{align*}
where $\lambda_{ij}, \theta_{ij}, \eta_{ij}$ are defined as above. It follows from the previous analysis that $\varphi_{ij}$ is a homomorphism from $(\mathcal{A}_n^{\star},^{\sigma_1})$ to $(\s_3,^{\sigma_1})$. In fact, the following result holds.


\begin{lemma}
\label{lem:varphi_equiv_inv2_sn}
Let $\bfu, \bfv \in \mathcal{A}_n^{\star}$. Then $\bfu \equiv_{\m} \bfv$ if and only if $\varphi_{ij}(\bfu) \equiv_{\m} \varphi_{ij}(\bfv)$ for all $1 \leq i < j \leq n$.
\end{lemma}
\begin{proof}
To prove the necessity, it suffices to show that for any $i,j$ with $1 \leq i < j \leq n$, if $\bfu\equiv_{\m}\bfv$, then $\varphi_{ij}(\bfu)\equiv_{\m}\varphi_{ij}(\bfv)$.
If $\bfu\equiv_{\mathsf{m}}\bfv$, then $\ev(\bfu)=\ev(\bfv)$ by Proposition \ref{pro:Sij}.
It follows from the definition of $\varphi_{ij}$ that  $\ev(\varphi_{ij}(\bfu))=\ev(\varphi_{ij}(\bfv))$. In the following, we prove that
$\ip(\varphi_{ij}(\bfu))=\ip(\varphi_{ij}(\bfv))$ and $\fp(\varphi_{ij}(\bfu))=\fp(\varphi_{ij}(\bfv))$, that is,
for any $h, k$ with $1\leq h<k\leq3$, ${_1h} \prec_{\varphi_{ij}(\bfu)} {_1k}$ if and only if ${_1h}\prec_{\varphi_{ij}(\bfv)} {_1k}$ and ${_\infty h}\prec_{\varphi_{ij}(\bfu)} {_\infty k}$ if and only if ${_\infty h}\prec_{\varphi_{ij}(\bfv)} {_\infty k}$.

{\bf Case 1.} $\varphi_{ij}=\lambda_{ij}$. If $\bfu\equiv_{\m}\bfv$, then ${_11} \prec_{\bfu} {_1n}$ if and only if ${_11}\prec_{\bfv} {_1n}$ and ${_\infty 1}\prec_{\bfu} {_\infty n}$ if and only if ${_\infty 1}\prec_{\bfv} {_\infty n}$ by Proposition \ref{pro:Sij}.
It follows from the definition of $\varphi_{ij}$ that  ${_11} \prec_{\varphi_{1n}(\bfu)} {_13} $ if and only if ${_11} \prec_{\varphi_{1n}(\bfv)} {_13}$ and ${_\infty 1}\prec_{\varphi_{1n}(\bfu)} {_\infty 3}$ if and only if ${_\infty 1}\prec_{\varphi_{1n}(\bfv)} {_\infty 3}$.

{\bf Case 2.} $\varphi_{ij}=\theta_{ij}$.
If $\bfu\equiv_{\m}\bfv$, then  for any $s<t\in \con(\bfu)$, ${_1}s \prec_{\bfu} {_1}t$ if and only if ${_1}s \prec_{\bfv} {_1}t$ and ${_\infty} s \prec_{\bfu} {_\infty} t$ if and only if ${_\infty} s\prec_{\bfv} {_\infty} t$ by Proposition \ref{pro:Sij}.
It follows from the definition of $\varphi_{ij}$
that for any $h, k$ with $1\leq h<k\leq3$, ${_1h} \prec_{\varphi_{ij}(\bfu)} {_1k}$ if and only if ${_1h}\prec_{\varphi_{ij}(\bfv)} {_1k}$ and ${_\infty h}\prec_{\varphi_{ij}(\bfu)} {_\infty k}$ if and only if ${_\infty h}\prec_{\varphi_{ij}(\bfv)} {_\infty k}$.

{\bf Case 3.} $\varphi_{ij}=\eta_{ij}$. If $\bfu\equiv_{\m}\bfv$, then  ${_1i} \prec_{\bfu} {_1j}$ if and only if ${_1i}\prec_{\bfv} {_1j}$ and ${_\infty i}\prec_{\bfu} {_\infty j}$ if and only if ${_\infty i}\prec_{\bfv} {_\infty j}$ by Proposition \ref{pro:Sij}. It follows from the definition of $\varphi_{ij}$ that  ${_11} \prec_{\varphi_{ij}(\bfu)} {_13} $ if and only if ${_11} \prec_{\varphi_{ij}(\bfv)} {_13}$ and ${_\infty 1}\prec_{\varphi_{ij}(\bfu)} {_\infty 3}$ if and only if ${_\infty 1}\prec_{\varphi_{ij}(\bfv)} {_\infty 3}$.

The sufficiency can be obtained symmetrically.
\end{proof}

For each $n \geq 4$, consider the map
\[
\phi_n : (\sn,^{\sigma_1}) \longrightarrow \prod\limits_{I_n} (\s_3,^{\sigma_1})
\]
whose $(i,j)$-th component is given by $\varphi_{ij}([\bfw]_{\m})$ for $\bfw \in \mathcal{A}_n^{\star}$ and $(i,j) \in I_n$.
		
\begin{lemma}
\label{lem:staln_embed_inv2_stal3}
The map $\phi_n$ is an embedding.
\end{lemma}
	
\begin{proof}
It follows from the definition of $\phi_n$ and Lemma~\ref{lem:varphi_equiv_inv2_sn} that $\phi_n$ is a homomorphism and for any $\bfu, \bfv \in \mathcal{A}_n^{\star}$, $\bfu \equiv_{\m} \bfv$ if and only if $\phi_n([\bfu]_{\m}) = \phi_n([\bfv]_{\m})$. Therefore $\phi_n$ is an embedding.
\end{proof}

\begin{theorem}\label{thm:stal-id-inv2-stal3}
For any $ n \geq 3$, $(\s_n,^{\sigma_1})$ and $(\s_3,^{\sigma_1})$ satisfy exactly the same identities.
\end{theorem}
\begin{proof}
By Theorem \ref{lem:staln_embed_inv2_stal3}, $(\sn,^{\sigma_1})$  can be embedded into a direct product of $\binom{n}{2}$ copies of $(\s_3,^{\sigma_1})$.
Thus,  for any $n> 3$, $(\sn,^{\sigma_1})$ is in the variety generated by $(\s_3,^{\sigma_1})$. Since $(\s_3,^{\sigma_1})$ is a submonoid of $(\sn,^{\sigma_1})$ for any $n > 3$, $(\s_3,^{\sigma_1})$ is in the variety generated by $(\s_n,^{\sigma_1})$. Therefore they satisfy exactly the same identities.
\end{proof}
%

\subsection{Characterizations of identities and identity basis}
\begin{theorem}\label{lem:inv2ids2}
An identity $\bfu\approx \bfv$ holds in $(\s_2,^{\sigma_1})$ if and only if
\begin{enumerate}[\rm(i)]
  \item $\bfu\approx \bfv$ is balanced.
  \item  for any $x,y \in \con(\bfu\bfv)$ with $x, x^*\neq y$, the first [resp. last] variable of $\bfu[x,y]$ is the same as that of $\bfv[x,y]$.
\end{enumerate}
\end{theorem}
\begin{proof}
Let $\bfu\approx \bfv$ be any identity of $(\s_2,^{\sigma_1})$. Suppose that there exists some $x\in \mathcal{X}$ such that $\occ(x, \bfu)\neq \occ(x, \bfv)$. Let $\theta_1$ be the substitution that maps $x$ to $[1]_{\m}$ and any other variable to $[\varepsilon]_{\m}$. Then $\ev(\theta_1(\bfu))\neq \ev(\theta_1(\bfv))$, a contradiction. Therefore the condition (i) holds.

Note that the involution submonoid of $(\s_2,^{\sigma_1})$ generated by $[12]_{\m}$ and $[21]_{\m}$ is isomorphic to $(\s_2,^{*_1})$. Thus condition (ii) holds by Theorem \ref{lem:inv1ids2}.

Conversely, let $\bfu\approx \bfv$ be any identity satisfying (i) and (ii) and $\phi$ be any substitution from $\mathcal{X}$ to $\s_2$.
It follows from (i) that $\ev(\phi(\bfu))=\ev(\phi(\bfv))$. If $\con(\phi(\bfu))=\{1\}$ or $\{2\}$, then by (i) we have $\phi(\bfu)=\phi(\bfv)$.
If $\con(\phi(\bfu))=\{1,2 \}$, to show $\phi(\bfu)=\phi(\bfv)$, it suffices to show that $\ip(\phi(\bfu))=\ip(\phi(\bfv))$ and $\fp(\phi(\bfu))=\fp(\phi(\bfv))$. By symmetry, we only need to show that $\ip(\phi(\bfu))=\ip(\phi(\bfv))$.

If ${_11}\prec_{\phi(\bfu)} {_12}$, assume that $\phi(x)\neq [\varepsilon]_{\m}$ for any $x\in \con(\bfu)$ by Remark \ref{rem:delete}, then $\bfu$ can be written into the form $\bfu = z\bfu_1$ satisfying $\phi(z)$ starts with $1$.  It follows from (ii) that $\bfv = z\bfv_1$.
Thus ${_11}\prec_{\phi(\bfu)} {_12}$  implies ${_11}\prec_{\phi(\bfv)} {_12}$, and by symmetry we have that ${_11}\prec_{\phi(\bfu)} {_12}$  if and only if ${_11}\prec_{\phi(\bfv)} {_12}$.
 Therefore  $\ip(\phi(\bfu))=\ip(\phi(\bfv))$.
\end{proof}

\begin{theorem}\label{lem:inv2ids3}
An identity $\bfu\approx \bfv$ holds in $(\s_3,^{\sigma_1})$ if and only if
\begin{enumerate}[\rm(i)]
  \item $\bfu\approx \bfv$ is balanced.
  \item  for any $x,y \in \con(\bfu\bfv)$ with $x, x^* \neq y$
   \begin{enumerate}[\rm(a)]
      \item  if $\bfu[x,y]\in x^{\alpha}x^*F_{\mathsf{inv}}^{\varepsilon}(\{x, y\})$ for some $\alpha\geq 1$, then $\bfv[x,y]\in x^{\beta}x^*F_{\mathsf{inv}}^{\varepsilon}(\{x, y\})$ for some $\beta\geq 1$.

      \item if $\bfu[x,y]\in F_{\mathsf{inv}}^{\varepsilon}(\{x, y\})x^*x^{\alpha}$ for some $\alpha\geq 1$, then $\bfv[x,y]\in F_{\mathsf{inv}}^{\varepsilon}(\{x, y\})x^*x^{\beta}$ for some $\beta\geq 1$.
      \item  if $\bfu[x,y]\in x^{\alpha}yF_{\mathsf{inv}}^{\varepsilon}(\{x, y\})$ for some $\alpha\geq 1$, then $\bfv[x,y]\in x^{\beta}yF_{\mathsf{inv}}^{\varepsilon}(\{x, y\})$ for some $\beta\geq 1$.

      \item if $\bfu[x,y]\in F_{\mathsf{inv}}^{\varepsilon}(\{x, y\})yx^{\alpha}$ for some $\alpha\geq 1$, then $\bfv[x,y]\in F_{\mathsf{inv}}^{\varepsilon}(\{x, y\})yx^{\beta}$ for some $\beta\geq 1$.
   \end{enumerate}
\end{enumerate}
\end{theorem}
\begin{proof}
Let $\bfu\approx \bfv$ be any identity of $(\s_3,^{\sigma_1})$.
Note that the involution submonoid of $(\s_3,^{\sigma_1})$ generated by $[1]_{\m}$ and $[3]_{\m}$ is isomorphic to $(\s_2,^{\sigma_1})$. Thus condition (i) holds by Theorem \ref{lem:inv2ids2}.

If $\bfu[x,y]\in x^{\alpha}x^*F_{\mathsf{inv}}^{\varepsilon}(\{x, y\})$ for some $\alpha\geq 1$, then $\bfv[x,y]\in x^{\beta}zF_{\mathsf{inv}}^{\varepsilon}(\{x, y\})$ for some $z\neq x, \beta\geq 1$ by Theorem \ref{lem:inv2ids2}.
Suppose that $z\neq x^*$. Let $\theta_1$ be the substitution such that $x\mapsto [1]_{{\s}_3}, z\mapsto [2]_{{\s}_3}$. Then
\begin{align*}
\theta_1(\bfu[x,y])=[1^{\alpha}3]_{\m}\cdot [\bfs]_{\m}\neq [1^{ \beta}2]_{\m}\cdot [\bft]_{\m}=\theta_1(\bfv[x,y])
\end{align*}
where $\bfs, \bft \in \mathcal{A}_2^{\star}$, a contradiction. Thus $\bfv[x,y]\in x^{\beta}x^*F_{\mathsf{inv}}^{\varepsilon}(\{x, y\})$ for some $\beta\geq 1$, and so the condition (iia) holds. The condition (iib) holds by symmetry.

If $\bfu[x,y]\in x^{\alpha}yF_{\mathsf{inv}}^{\varepsilon}(\{x, y\})$ for some $y \neq x, x^*, \alpha\geq 1$, then $\bfv[x,y]\in x^{\beta}zF_{\mathsf{inv}}^{\varepsilon}(\{x, y\})$ for some $z\neq x, \beta\geq 1$ by Theorem \ref{lem:inv2ids2}.
It follows from the above argument that $z = y$. Thus $\bfv[x,y]\in x^{\beta}yF_{\mathsf{inv}}^{\varepsilon}(\{x, y\})$ for some $\beta\geq 1$, and so the condition (iic) holds. The condition (iid) holds by symmetry.

Conversely, let $\bfu\approx \bfv$ be any identity satisfying (i) and (ii) and $\phi$ be any substitution from $\mathcal{X}$ to $\s_3$.
It follows from (i) that $\ev(\phi(\bfu))=\ev(\phi(\bfv))$. To show $\phi(\bfu)=\phi(\bfv)$, it suffices to show that $\ip(\phi(\bfu))=\ip(\phi(\bfv))$ and $\fp(\phi(\bfu))=\fp(\phi(\bfv))$. By symmetry, we only need to show that $\ip(\phi(\bfu))=\ip(\phi(\bfv))$.
We may assume that $\phi(x)\neq [\varepsilon]_{\m}$ for any $x\in \con(\bfu)$ by Remark \ref{rem:delete}.

If ${_11}\prec_{\phi(\bfu)} {_12}$, then $\bfu$ can be written into the form $\bfu = \bfu_1z\bfu_2$ for some possibly empty words $\bfu_1, \bfu_2$
satisfying $\{1\}\subseteq\con(\phi(\bfu_1))\subseteq\{1, 3\}$ and $2 \in \con(\phi(z))$.  Note that $z, z^*\not\in\con(\bfu_1)$. It follows from (iia) and (iic) that $\bfv = \bfv_1z\bfv_2$ for some possibly empty words $\bfv_1, \bfv_2$ satisfying $\con(\bfu_1)=\con(\bfv_1)$.
Therefore ${_11}\prec_{\phi(\bfu)} {_12}$ implies ${_11}\prec_{\phi(\bfv)} {_12}$, and by symmetry we have that ${_11}\prec_{\phi(\bfu)} {_12}$ if and only if ${_11}\prec_{\phi(\bfv)} {_12}$.

If ${_11}\prec_{\phi(\bfu)} {_13}$, then $\bfu$ can be written into the form $\bfu = \bfu_1z\bfu_2$ for some possibly empty words $\bfu_1, \bfu_2$
satisfying $\{1\}\subseteq\con(\phi(\bfu_1))\subseteq\{1, 2\}$ and $3 \in \con(\phi(z))$.  Note that for any $x\in \con(\bfu_1)$, $x^*\in \con(\bfu_1)$ only if $\con(\phi(x))=\{2\}$, $z\not\in\con(\bfu_1)$, and $z^*\in\con(\bfu_1)$ only if $\con(\phi(z))=\{3\}$ or $\{2, 3\}$. Let $\mathcal{X}\subseteq \con(\bfu_1)$ be the set satisfying for any $x\in \mathcal{X}$, $\con(\phi(x))\neq \{2\}$.  It follows from (iia) and (iic) that $\bfv = \bfv_1z\bfv_2$ for some possibly empty words $\bfv_1, \bfv_2$ satisfying $\con(\bfu_1[\mathcal{X}])=\con(\bfv_1[\mathcal{X}])$.
Therefore ${_11}\prec_{\phi(\bfu)} {_13}$ implies ${_11}\prec_{\phi(\bfv)} {_13}$, and by symmetry we have that ${_11}\prec_{\phi(\bfu)} {_13}$ if and only if ${_11}\prec_{\phi(\bfv)} {_13}$.

If ${_13}\prec_{\phi(\bfu)} {_12}$, then $\bfu$ can be written into the form $\bfu = \bfu_1z\bfu_2$ for some possibly empty words $\bfu_1, \bfu_2$
satisfying $\{3\}\subseteq\con(\phi(\bfu_1))\subseteq\{1, 3\}$ and $2 \in \con(\phi(z))$.  Note that $z, z^*\not\in\con(\bfu_1)$. It follows from (iia) and (iic) that $\bfv = \bfv_1z\bfv_2$ for some possibly empty words $\bfv_1, \bfv_2$ satisfying $\con(\bfu_1)=\con(\bfv_1)$.
Therefore ${_13}\prec_{\phi(\bfu)} {_12}$ implies ${_13}\prec_{\phi(\bfv)} {_12}$, and by symmetry we have that ${_13}\prec_{\phi(\bfu)} {_12}$ if and only if  ${_13}\prec_{\phi(\bfv)} {_12}$.
\end{proof}

An immediate consequence of Theorems \ref{lem:inv2ids2} and \ref{lem:inv2ids3} is the following.

\begin{corollary}
The decision problems ${\textsc{Check-Id}}(\s_2,^{\sigma_1})$ and ${\textsc{Check-Id}}(\s_3,^{\sigma_1})$
belong to the complexity class $\mathsf{P}$.
\end{corollary}

In the end of this section, we prove that $(\s_2,^{\sigma_1})$ and $(\s_3,^{\sigma_1})$ are finitely based by giving a finite identity basis for them.

\begin{theorem}\label{thm:FBS2inv2}
The variety generated by  $(\s_2,^{\sigma_1})$ is defined by the  identities \eqref{id: inv} and
\begin{gather}
x^{\circledast_1}hxykx^{\circledast_2}\approx x^{\circledast_1}hyxkx^{\circledast_2}, \tag{4.2} \label{4.2} \\
x^{\circledast_1}hyxky^{\circledast_2} \approx x^{\circledast_1}hxyky^{\circledast_2}, \tag{4.3} \label{4.3}
\end{gather}
where $\circledast_1, \circledast_2\in\{1, *\}$.
\end{theorem}
\begin{proof}
Clearly, $(\s_2,^{\sigma_1})$ satisfies the identities \eqref{4.2}--\eqref{4.3} by Theorem~\ref{lem:inv2ids2}.
It suffices to show that each non-trivial identity satisfied by $(\s_2,^{\sigma_1})$ can be deduced from \eqref{4.2}--\eqref{4.3}. Note that any identity satisfied by $(\s_2,^{\sigma_1})$ is balanced.
Let $\Sigma$ be the set of all non-trivial balanced identities satisfied by $(\s_2,^{\sigma_1})$ but  not  deducible from \eqref{4.2} and \eqref{4.3}.
The following proof is exactly the same as the proof in Theorem \ref{thm:FBS2inv1}.
\end{proof}

\begin{theorem}\label{thm:FBS3inv2}
The variety generated by  $(\s_3,^{\sigma_1})$ is defined by the  identities \eqref{id: inv}, \eqref{id:xhxyty} and
\begin{gather}
xhyxkx^*sy^* \approx xhxykx^*sy^*,\quad xhyxky^*sx^* \approx xhxyky^*sx^*,\label{e1}\\
x^*sy^*hyxky \approx x^*sy^*hxyky,\quad y^*sx^*hyxky \approx y^*sx^*hxyky, \label{e2}\\
x^*sy^*hyxkx^*ty^* \approx x^*sy^*hxykx^*ty^*,\quad x^*sy^*hyxky^*tx^* \approx x^*sy^*hxyky^*tx^*. \label{e3}
\end{gather}
\end{theorem}
\begin{proof}
Clearly, $(\s_3,^{\sigma_1})$ satisfies the identities \eqref{id:xhxyty} and \eqref{e1}--\eqref{e3} by Theorem~\ref{lem:inv2ids3}.
It suffices to show that each non-trivial identity satisfied by $(\s_3,^{\sigma_1})$ can be deduced from \eqref{id:xhxyty} and \eqref{e1}--\eqref{e3}. Note that any identity satisfied by $(\s_3,^{\sigma_1})$ is balanced.
Let $\Sigma$ be the set of all non-trivial balanced identities satisfied by $(\s_3,^{\sigma_1})$ but  not  deducible from \eqref{id:xhxyty} and \eqref{e1}--\eqref{e3}. Suppose that $\Sigma\neq \emptyset$. Note that any balanced identity $\bfu\approx \bfv$ in $\Sigma$ can be written
uniquely in the form $\bfu' a\bfw \approx \bfv' b\bfw$ where $a,b$ are distinct variables and $|\bfu'|=|\bfv'|$. Choose an identity, say $\bfu\approx \bfv$, from $\Sigma$ such that when it is written as $\bfu' a\bfw\approx \bfv' b\bfw$,
the lengths of the words $\bfu'$ and $\bfv'$ are as short as possible. Since $\bfu\approx \bfv$ is balanced, we have $\con(\bfu'a)=\con(\bfv'b)$ and $\occ(x, \bfu'a)=\occ(x, \bfv'b)$ for any $x \in \con(\bfu'a)$.

Since $\con(\bfu'a)=\con(\bfv'b)$, we have $b\in \con(\bfu')$, and so $\bfu'a=\bfu_1 bc\bfu_2$ with $b\neq c$ and $b\not\in\con(\bfu_2)$.  Since $\con(\bfu'a)=\con(\bfv'b)$, we have $c\in \con(\bfv')$, and so $\bfv'=\bfv_1 c\bfv_2$ with $c\not\in\con(\bfv_2)$. Thus $\bfu=\bfu_1 bc\bfu_2\bfw$ and $\bfv=\bfv_1 c\bfv_2b\bfw$.
On the one hand, we claim that $b\in\con(\bfu_1)$ or $b^*\in\con(\bfu_1)$ when $b^*=c$ and that at least one of the three cases $b\in\con(\bfu_1)$, $c\in\con(\bfu_1)$, or $b^*, c^*\in\con(\bfu_1)$ is true when $b^*\neq c$. First we consider the case $b^*=c$. If $b, b^*\not\in\con(\bfu_1)$, then $\occ(b, \bfu'a)=1$, and so $\bfu[b]$ starts with $b$ but $\bfv[b]$ does not, which contradicts (iia) of Theorem~\ref{lem:inv2ids3}. Next we consider the case $b^*\neq c$. If $b, c, b^*, c^* \not\in\con(\bfu_1)$, then $\occ(b, \bfu'a)=1$, and so $\bfu[b,c]$ starts with $bc$ but $\bfv[b,c]$ does not, which contradicts (iic) of Theorem~\ref{lem:inv2ids3}. If only $b^*\in\con(\bfu_1)$, then $\occ(b, \bfu'a)=1$, and so $\bfu[b,c]$ starts with $b^*b$ but $\bfv[b,c]$ does not, which contradicts (iia) of Theorem~\ref{lem:inv2ids3}. If only $c^*\in\con(\bfu_1)$, then $\occ(b, \bfu'a)=1$, and so $\bfu[b,c]$ starts with $c^*b$ but $\bfv[b,c]$ does not, which contradicts (iic) of Theorem~\ref{lem:inv2ids3}.
On the other hand,
we claim that at $b\in\con(\bfu_2\bfw)$ or $b^*\in\con(\bfu_2\bfw)$ when $b^*=c$ and that at least one of the three cases $b\in\con(\bfu_2\bfw)$, $c\in\con(\bfu_2\bfw)$, or $b^*,c^*\in\con(\bfu_2\bfw)$ is true when $b^*\neq c$. First we consider the case $b^*=c$.  If $b, b^*\not\in\con(\bfu_2\bfw)$, then $\bfu[b]$ ends with $b^*$ but $\bfv[b]$ does not, which contradicts (iib) of Theorem~\ref{lem:inv2ids3}. Next we consider the case $b^*\neq c$. If $b, c, b^*, c^*\not\in\con(\bfu_2\bfw)$, then $\bfu[b,c]$ ends with $bc$ but $\bfv[b,c]$ does not, which contradicts (iid) of Theorem~\ref{lem:inv2ids3}.
If only $b^*\in\con(\bfu_2\bfw)$, then $\bfu[b,c]$ ends with $cb^*$ but $\bfv[b,c]$ does not, which contradicts (iid) of Theorem~\ref{lem:inv2ids3}.
If only $c^*\in\con(\bfu_2\bfw)$, then $\bfu[b,c]$ ends with $cc^*$ but $\bfv[b,c]$ does not, which contradicts (iib) of Theorem~\ref{lem:inv2ids3}.
Therefore, the
identities \eqref{id:xhxyty} and \eqref{e1}--\eqref{e3} can be used to convert $\bfu_1 bc\bfu_2 \bfw$ into $\bfu_1cb\bfu_2\bfw$. By repeating this process, the word $\bfu=\bfu'a\bfw=\bfu_1bc\bfu_2 \bfw$ can be converted into the word $\bfu_1c\bfu_2b \bfw$ by the
identities \eqref{id:xhxyty} and \eqref{e1}--\eqref{e3}.

Clearly the identities $\bfu_1c\bfu_2b\bfw \approx \bfu'a\bfw \approx \bfv'b\bfw$ hold in $(\s_3, ^{\sigma_1})$. Note that $|\bfu_1c\bfu_2b\bfw|=|\bfu'a\bfw|=|\bfv'b\bfw|$ and words in the identity $\bfu_1c\bfu_2b\bfw \approx \bfv'b\bfw$ have a longer common suffix than words in the identity $\bfu'a\bfw \approx \bfv'b\bfw$. Hence $\bfu_1c\bfu_2b\bfw \approx \bfv'b\bfw \notin \Sigma$ by the minimality assumption on the lengths of $\bfu'$ and $\bfv'$, that is, $\bfu_1c\bfu_2b\bfw \approx \bfv'b\bfw$ can be deducible from \eqref{id:xhxyty} and \eqref{e1}--\eqref{e3}. We have shown that $\bfu'a\bfw \approx \bfu_1c\bfu_2b\bfw$ can be deducible from \eqref{id:xhxyty} and \eqref{e1}--\eqref{e3}. Hence $\bfu'a\bfw \approx \bfv'b\bfw$ can be deducible from \eqref{id:xhxyty} and \eqref{e1}--\eqref{e3}, which contradicts $\bfu'a\bfw \approx \bfv'b\bfw \in \Sigma$. Therefore $\Sigma=\emptyset$.
\end{proof}

\section{Identities of $(\sn, ^{\sigma_i})$ for $2\leq i\leq \lfloor\frac{n}{2}\rfloor$}\label{sec:5.2}%
In this section,  for $2\leq i\leq \lfloor\frac{n}{2}\rfloor$, we investigate the identities of $\sn$ under the involution $^{\sigma_i}$ determined by the permutation $\sigma_i$ on $\mathcal{A}_n$ that can be decomposed into $i$ disjoint $2$-cycles.  Recall that, for each $i$ with $2\leq i\leq \lfloor\frac{n}{2}\rfloor$,
each permutation on $\mathcal{A}_n$ that can be decomposed into $i$ disjoint $2$-cycles induces an involution on $\s_n$.
Involution semigroups of $\s_n$ under each of these involutions are isomorphic. Let $\sigma\in \{\sigma_2, \sigma_3, \dots, \sigma_{\lfloor\frac{n}{2}\rfloor}\}$. For convenience, we only consider the involution $^\sigma$ determined by the following permutation $\sigma$:
\begin{gather*}
\sigma(1)=n, \quad \sigma(n)=1, \\
\sigma(2)=n-1,  \quad  \sigma(n-1)=2,\\
\vdots\\
\sigma(i)=n-(i-1), \quad \sigma(n-(i-1))=i,  \\
\sigma(i+1)=i+1, \quad \sigma(i+2)=(i+2),\cdots, \sigma(n-i)=n-i
\end{gather*}
for $i\geq2, n\geq 4$.
We prove that for all finite $n\geq 4$, $(\sn, ^{\sigma})$ satisfy the same identities. Further,
we give characterizations of the identities satisfied by them; based on these results,
the finite basis problem and identity checking problem for $(\sn, ^{\sigma})$ are solved.

For any $([\bfu]_{\m}, [\bfv]_{\m})\in \s_2\times \s_2$, define a unary operation $^{\rho}$ on $\s_2\times \s_2$ by  $([\bfu]_{\m}, [\bfv]_{\m})^{\rho}= ([\bfv]_{\m}^{\sigma_1}, [\bfu]_{\m}^{\sigma_1})$ in which $^{\sigma_1}$ on $\s_2$ is defined in Section 2. Thus $^\rho$ is an involution on $\s_2\times \s_2$.
\subsection{Embedding into a direct product of copies of $(\s_2\times \s_2,^{\rho})$}
For $n \geq 4$, we show that $(\sn,^{\sigma})$  can be embedded into a direct product of copies of $(\s_2\times \s_2,^{\rho})$.
For any $i<j \in \mathcal{A}_n$ with $n\geq 4$, there are six cases about the order of $i,j, i^{\sigma}, j^{\sigma}$ in $\mathcal{A}_n$: $i^{\sigma}=j$, $i^{\sigma}=i<j=j^{\sigma}$, $i<j\leq j^{\sigma}<i^{\sigma}$, $j^{\sigma}<i^{\sigma}\leq i<j$, $i<j^{\sigma}<j<i^{\sigma}$, $j^{\sigma}<i<i^{\sigma}<j$.
For any $i<j \in \mathcal{A}_n$ with $n\geq 4$, define a map $\varphi_{ij}$ from $\mathcal{A}_n^{\star}$ to $\s_2\times \s_2$ which can be determined by the following four cases according to the order of $i, i^{\sigma}, j, j^{\sigma}$ in $\mathcal{A}_n$.

{\bf Case 1.} $i^{\sigma} = j$. Define a map $\lambda_{ij}: \mathcal{A}_n \rightarrow \s_2\times \s_2$ given by
\begin{align*}
k &\mapsto \begin{cases}
([1]_{\m}, [1]_{\m}), & \text{if}\ k = i,\\
([2]_{\m}, [2]_{\m}), & \text{if}\ k = j,\\
([\varepsilon]_{\m}, [\varepsilon]_{\m}), & \text{otherwise}.
\end{cases}
\end{align*}
This map can be extended to a homomorphism $\lambda_{ij}: \mathcal{A}_n^{\star} \rightarrow \s_2\times \s_2$. Further, $\lambda_{ij}$
is also a homomorphism from $(\mathcal{A}_n^{\star},^{\sigma})$ to $(\s_2\times \s_2,^{\rho})$. This is because for any $k \in \mathcal{A}_n$, $\lambda_{ij}(k^{\sigma})=(\lambda_{ij}(k))^{\rho}$ by
{\setlength{\arraycolsep}{0.5pt}
\[
\left\{
\begin{array}{llll}
\lambda_{ij}(k^{\sigma})=([2]_{\m}, [2]_{\m})&=(\lambda_{ij}(k))^{\rho}, & \quad \text{if}\ k = i,\\
\lambda_{ij}(k^{\sigma})=([1]_{\m}, [1]_{\m})&=(\lambda_{ij}(k))^{\rho},  & \quad\text{if}\ k = j,\\
\lambda_{ij}(k^{\sigma})=([\varepsilon]_{\m}, [\varepsilon]_{\m})&=(\lambda_{ij}(k))^{\rho},   & \quad\text{otherwise}.
\end{array}\right.
\]}

{\bf Case 2.} $i^{\sigma}=i<j=j^{\sigma}$.  Define a map $\eta_{ij}: \mathcal{A}_n \rightarrow \s_2\times \s_2$ by
\begin{align*}
k \mapsto \begin{cases}
([12]_{\m}, [12]_{\m}), & \text{if}\ k = i,\\
([21]_{\m}, [21]_{\m}), & \text{if}\ k = j,\\
([\varepsilon]_{\m}, [\varepsilon]_{\m}), & \text{otherwise.}
\end{cases}
\end{align*}
This map can be extended to a homomorphism $\eta_{ij}: \mathcal{A}_n^{\star} \rightarrow \s_2\times \s_2$. Further, $\eta_{ij}$
is also a homomorphism from $(\mathcal{A}_n^{\star},^{\sigma})$ to $(\s_2\times \s_2,^{\rho})$. This is because for any $k \in \mathcal{A}_n$, $\eta_{ij}(k^{\sigma})=(\eta_{ij}(k))^{\rho}$ by
\begin{align*}
\begin{cases}
\eta_{ij}(k^{\sigma})=([12]_{\m}, [12]_{\m})= (\eta_{ij}(k))^{\rho}, & \text{if}\ k = i,\\
\eta_{ij}(k^{\sigma})=([21]_{\m}, [21]_{\m})=(\eta_{ij}(k))^{\rho},  & \text{if}\ k =j,\\
\eta_{ij}(k^{\sigma})=([\varepsilon]_{\m}, [\varepsilon]_{\m})\,\,\,=(\eta_{ij}(k))^{\rho},  & \text{otherwise}.
\end{cases}
\end{align*}

{\bf Case 3.} $i<j\leq j^{\sigma}<i^{\sigma}$ or $j^{\sigma}<i^{\sigma}\leq i<j$. For convenience, let $i_1=i, i_2=j, i_3=j^{\sigma}$ and $i_4=i^{\sigma}$ when $i<j\leq j^{\sigma}<i^{\sigma}$ and $i_1=j^{\sigma}, i_2=i^{\sigma},i_3=i$ and $i_4=j$ when $j^{\sigma}<i^{\sigma}\leq i<j$. Define a map $\theta_{ij}: \mathcal{A}_n \rightarrow \s_2\times \s_2$ by
\begin{align*}
k \mapsto \begin{cases}
([1]_{\m}, [\varepsilon]_{\m}), & \text{if}\ k = i_1,\\
([2]_{\m}, [\varepsilon]_{\m}), & \text{if}\ k = i_2\neq i_3,\\
([2]_{\m}, [1]_{\m}), & \text{if}\ k = i_2=i_3,\\
([\varepsilon]_{\m}, [1]_{\m}), & \text{if}\ k = i_3\neq i_2,\\
([\varepsilon]_{\m}, [2]_{\m}), & \text{if}\ k = i_4,\\
([\varepsilon]_{\m}, [\varepsilon]_{\m}),  & \text{otherwise.}
\end{cases}
\end{align*}
This map can be extended to a homomorphism $\theta_{ij}: \mathcal{A}_n^{\star} \rightarrow \s_2\times \s_2$. Further, $\theta_{ij}$ is also a homomorphism from $(\mathcal{A}_n^{\star},^{\sigma})$ to $(\s_2\times \s_2,^{\rho})$. This is because for any $k \in \mathcal{A}_n$, $\theta_{ij}(k^{\sigma})=(\theta_{ij}(k))^{\rho}$ by
{\setlength{\arraycolsep}{0.5pt}
\begin{align*}
\left\{
\begin{array}{llllllll}
\theta_{ij}(k^{\sigma})&=([\varepsilon]_{\m}, [2]_{\m}) &=(\theta_{ij}(k))^{\rho}, &\quad \text{if}~k = i_1,\\
\theta_{ij}(k^{\sigma})&=([\varepsilon]_{\m}, [1]_{\m}) &=(\theta_{ij}(k))^{\rho}, &\quad \text{if}\ k= i_2, i_2\neq i_3,\\
\theta_{ij}(k^{\sigma})&=([2]_{\m}, [1]_{\m})&=(\theta_{ij}(k))^{\rho}, &\quad \text{if}\ k = i_2=i_3,\\
\theta_{ij}(k^{\sigma})&=([2]_{\m}, [\varepsilon]_{\m})&=(\theta_{ij}(k))^{\rho}, &\quad \text{if}\ k= i_3, i_2\neq i_3,\\
\theta_{ij}(k^{\sigma})&=([1]_{\m}, [\varepsilon]_{\m})&=(\theta_{ij}(k))^{\rho}, &\quad \text{if}\ k =i_4,\\
\theta_{ij}(k^{\sigma})&=([\varepsilon]_{\m}, [\varepsilon]_{\m}) &=(\theta_{ij}(k))^{\rho}, &\quad\text{otherwise}.
\end{array}\right.
\end{align*}}

{\bf Case 4.} $i<j^{\sigma}<j<i^{\sigma}$ or $j^{\sigma}<i<i^{\sigma}<j$. For convenience, let $i_1=i, i_2=j^{\sigma}, i_3=j$ and $i_4=i^{\sigma}$ when $i<j^{\sigma}<j<i^{\sigma}$ and $i_1=j^{\sigma}, i_2=i,i_3=i^{\sigma}$ and $i_4=j$ when $j^{\sigma}<i<i^{\sigma}<j$. Define a map $\kappa_{ij}: \mathcal{A}_n \rightarrow \s_2\times \s_2$, by
\begin{align*}
k \mapsto \begin{cases}
([1]_{\m}, [\varepsilon]_{\m}), & \text{if}\ k = i_1,\\
([\varepsilon]_{\m}, [1]_{\m}), & \text{if}\ k = i_2,\\
([2]_{\m}, [\varepsilon]_{\m}), & \text{if}\ k = i_3,\\
([\varepsilon]_{\m}, [2]_{\m}), & \text{if}\ k = i_4,\\
([\varepsilon]_{\m}, [\varepsilon]_{\m}),  & \text{otherwise}.
\end{cases}
\end{align*}
Clearly, this map can be extended to a homomorphism from $\mathcal{A}_n^{\star}$ to $\s_2\times \s_2$. Further, $\kappa_{ij}$ is a homomorphism from $(\mathcal{A}_n^{\star},^{\sigma})$ to $(\s_2\times \s_2,^{\rho})$.
This is because for any $k \in \mathcal{A}_n$, $\kappa_{ij}(k^{\sigma})=(\kappa_{ij}(k))^{\rho}$ by
{\setlength{\arraycolsep}{0.5pt}
\begin{align*}
\left\{
\begin{array}{lllllllll}
\kappa_{ij}(k^{\sigma})&=([\varepsilon]_{\m}, [2]_{\m})&=(\kappa_{ij}(k))^{\rho}, &\quad \text{if}\ k = i_1,\\
\kappa_{ij}(k^{\sigma})&=([2]_{\m}, [\varepsilon]_{\m})&=(\kappa_{ij}(k))^{\rho}, & \quad\text{if}\ k = i_2,\\
\kappa_{ij}(k^{\sigma})&=([\varepsilon]_{\m}, [1]_{\m})&=(\kappa_{ij}(k))^{\rho},  &\quad \text{if}\ k = i_3,\\
\kappa_{ij}(k^{\sigma})&=([1]_{\m}, [\varepsilon]_{\m})&=(\kappa_{ij}(k))^{\rho}, & \quad \text{if}\ k=i_4,\\
\kappa_{ij}(k^{\sigma})&=([\varepsilon]_{\m}, [\varepsilon]_{\m})&=(\kappa_{ij}(k))^{\rho}, &\quad \text{otherwise}.
\end{array}\right.
\end{align*}}

Now we can define the map $\varphi_{ij}: \mathcal{A}^{\star}_n \rightarrow \s_2\times \s_2$ by
\begin{align*}
\varphi_{ij}=
\left\{
\begin{array}{ll}
\lambda_{ij} & \hbox{if $i^{\sigma}=j$,} \\
[0.1cm]
\eta_{ij} & \hbox{if $i^{\sigma}=i<j=j^{\sigma}$,} \\
[0.1cm]
\theta_{ij} & \hbox{if $i<j=j^{\sigma}<i^{\sigma}$ or $j^{\sigma}<i=i^{\sigma}<j$,} \\
[0.1cm]
\kappa_{ij} & \hbox{if $i<j^{\sigma}<j<i^{\sigma}$ or $j^{\sigma}<i<i^{\sigma}<j$}\\
\end{array}
\right.
\end{align*}
where $\lambda_{ij}, \eta_{ij}, \theta_{ij}, \kappa_{ij}$ are defined as above. It follows from the previous analysis that $\varphi_{ij}$ is a homomorphism from $(\mathcal{A}_n^{\star},^{\sigma})$ to $(\s_2\times \s_2,^{\rho})$.

\begin{lemma}
\label{lem:varphi_equiv_inv3_sn}
Let $\bfu, \bfv \in \mathcal{A}_n^{\star}$. Then $\bfu \equiv_{\m} \bfv$ if and only if $\varphi_{ij}(\bfu) \equiv_{\m} \varphi_{ij}(\bfv)$ for all $1 \leq i < j \leq n$.
\end{lemma}
	
\begin{proof}
To prove the necessity,
we only need to show that for any $i,j$ with $1 \leq i < j \leq n$, if $\bfu\equiv_{\m}\bfv$, then $\varphi_{ij}(\bfu)\equiv_{\m}\varphi_{ij}(\bfv)$.
It follows from Proposition \ref{pro:Sij} and the definition of $\varphi_{ij}$ that the evaluations of the first and the second components of $\varphi_{ij}(\bfu)$ and $\varphi_{ij}(\bfv)$ are the same respectively. In the following, we prove that in the first components $U_1$ and $V_1$ of $\varphi_{ij}(\bfu)$ and $\varphi_{ij}(\bfv)$ and in the second components $U_2$ and $V_2$ of $\varphi_{ij}(\bfu)$ and $\varphi_{ij}(\bfv)$,  ${_1} 1 \prec_{U_\ell} {_1} 2$ if and only if  ${_1} 1 \prec_{V_\ell} {_1} 2$  and ${_\infty} 1 \prec_{U_\ell} {_\infty} 2$ if and only if  ${_\infty} 1 \prec_{V_\ell} {_\infty} 2$ for $\ell\in \{1, 2\}$.

{\bf Case 1.} $\varphi_{ij}=\lambda_{ij}$ or $\eta_{ij}$. If $\bfu\equiv_{\m}\bfv$, then it follows from Proposition \ref{pro:Sij} that ${_1}i\prec_{\bfu} {_1}j$ if and only if ${_1}i\prec_{\bfv} {_1}j$ and ${_\infty}i\prec_{\bfu} {_\infty}j$ if and only if ${_\infty}i\prec_{\bfv} {_\infty}j$. It follows from the definition of $\varphi_{ij}$ that  ${_1} 1 \prec_{U_\ell} {_1} 2$ if and only if  ${_1} 1 \prec_{V_\ell} {_1} 2$  and ${_\infty} 1 \prec_{U_\ell} {_\infty} 2$ if and only if  ${_\infty} 1 \prec_{V_\ell} {_\infty} 2$ for $\ell\in \{1, 2\}$.

{\bf Case 2.} $\varphi_{ij}=\theta_{ij}$ or $\kappa_{ij}$.
If $\bfu\equiv_{\m}\bfv$, then it follows from Proposition \ref{pro:Sij} that ${_1}s\prec_{\bfu} {_1}t$ if and only if ${_1}s\prec_{\bfv} {_1}t$ and ${_\infty}s\prec_{\bfu} {_\infty}t$ if and only if ${_\infty}s\prec_{\bfv} {_\infty}t$  for $s=i,t=j$ or $s=j^{\sigma},t=i^{\sigma}$. It follows from the definition of $\varphi_{ij}$ that ${_1} 1 \prec_{U_\ell} {_1} 2$ if and only if  ${_1} 1 \prec_{V_\ell} {_1} 2$  and ${_\infty} 1 \prec_{U_\ell} {_\infty} 2$ if and only if  ${_\infty} 1 \prec_{V_\ell} {_\infty} 2$ for $\ell\in \{1, 2\}$.

The sufficiency can be obtained symmetrically.
\end{proof}

For each $n \geq 4$, consider the map
\[
\phi_n : (\sn,^{\sigma}) \longrightarrow \prod\limits_{I_n} (\s_2\times \s_2,^{\rho})
\]
whose $(i,j)$-th component is given by $\varphi_{ij}([\bfw]_{\m})$ for $\bfw \in \mathcal{A}_n^{\star}$ and $(i,j) \in I_n$.
		
\begin{lemma}
\label{lem:staln_embed_inv3_stal4}
The map $\phi_n$ is an embedding.
\end{lemma}
	
\begin{proof}
It follows from the definition of $\phi_n$ and Lemma~\ref{lem:varphi_equiv_inv3_sn} that $\phi_n$ is a homomorphism and for any $\bfu, \bfv \in \mathcal{A}_n^{\star}$, we have $\bfu \equiv_{\m} \bfv$ if and only if $\phi_n([\bfu]_{\m}) = \phi_n([\bfv]_{\m})$. Therefore $\phi_n$ is an embedding.
\end{proof}

%
%
	
For convenience, the variety generated by an {\insem} $(S,\op)$ is denoted by $\var (S,\op)$.
\begin{theorem}\label{thm:stal-id-inv3-stal4}
For any $n > 4$, $(\s_n,^{\sigma})$ and $(\s_4,^{\sigma})$ satisfy exactly the same identities.
\end{theorem}

\begin{proof}
It follows from Lemma \ref{lem:staln_embed_inv3_stal4} that
\[
\var(\s_4,^{\sigma})\subseteq \var(\s_5,^{\sigma}) \subseteq \cdots \subseteq\var (\s_2\times \s_2,^{\rho}).
\]
By Lemma \ref{lem:varphi_equiv_inv3_sn}, we have that $\varphi_{12}$ is a homomorphism form $(\s_4,^{\sigma})$ to $(\s_2\times \s_2,^{\rho})$. Since $\varphi_{12}$ maps all generators of $(\s_4,^{\sigma})$ to all generators of $(\s_2\times \s_2,^{\rho})$, $\varphi_{12}$ is surjective. Thus we have $\var(\s_2\times \s_2,^{\rho})\subseteq \var(\s_4,^{\sigma})$. Therefore the results holds.
\end{proof}

\subsection{Characterizations of identities and identity basis}

\begin{theorem}\label{lem:inv3ids4}
An identity $\bfu\approx \bfv$ holds in $(\s_4,^{\sigma})$ if and only if
\begin{enumerate}[\rm(i)]
  \item $\bfu\approx \bfv$ is balanced.
  \item  $\ip(\bfu)=\ip(\bfv), \fp(\bfu)=\fp(\bfv)$.
\end{enumerate}
\end{theorem}
\begin{proof}
Let $\bfu\approx \bfv$ be any identity of $(\s_4,^{\sigma})$.
Note that the involution submonoid of $(\s_4,^{\sigma})$ generated by $[1]_{\m}$ and $[4]_{\m}$ is isomorphic to $(\s_2,^{\sigma_1})$. Thus condition (i) holds by Theorem \ref{lem:inv2ids2}.
Suppose that there exist $x, y\in\con(\bfu)=\con(\bfv)$ such that the first $x$ precedes the first $y$ in $\bfu$ but the first $y$ precedes the first $x$ in $\bfv$.
If $y=x^*$, then Let $\phi_1$ be a substitution such that $x\mapsto [1]_{\m}$ and any other variable is mapped to $[\varepsilon]_{\m}$. Then ${_1}1\prec_{\phi_1(\bfu)} {_1}4$ but ${_1}4\prec_{\phi_1(\bfv)} {_1}1$, a contradiction.
If $y\neq x^*$, then let $\phi_2$ be a substitution such that $x\mapsto [1]_{\m}, y\mapsto [2]_{\m}$ and any other variable is mapped to $[\varepsilon]_{\m}$. Then ${_1}1\prec_{\phi_2(\bfu)} {_1}2$ but ${_1}2\prec_{\phi_2(\bfv)} {_1}1$, a contradiction. Therefore $\ip(\bfu)=\ip(\bfv)$.
Symmetrically, $\fp(\bfu)=\fp(\bfv)$.

Conversely, let $\bfu\approx \bfv$ be any balanced identity  satisfying $\ip(\bfu)=\ip(\bfv), \fp(\bfu)=\fp(\bfv)$ and $\phi$ be any substitution from $\mathcal{X}$ to $\s_4$.
Note that $\ev(\phi(\bfu))=\ev(\phi(\bfv))$ since $\bfu \approx \bfv$ is balanced. Now, to show $\phi(\bfu)=\phi(\bfv)$, it suffices to show that for any $1\leq i<j\leq 4$, ${_1}i\prec_{\phi(\bfu)} {_1}j$ if and only if  ${_1}i\prec_{\phi(\bfv)} {_1}j$ and ${_\infty}i\prec_{\phi(\bfu)} {_\infty}j$ if and only if  ${_\infty}i\prec_{\phi(\bfv)} {_\infty}j$. By symmetry, we only consider the former case.
If ${_1}i\prec_{\phi(\bfu)} {_1}j$, then $\bfu$ can be written into the form $\bfu = \bfu_1z\bfu_2$
satisfying $j\not\in \con(\phi(\bfu_1))$ and $j \in \con(\phi(z))$. Note that $z\not\in\con(\bfu_1)$. Thus, ${_1}i\prec_{\phi(\bfu_1z)} {_1}j$ if and only if  ${_1}i\prec_{\phi(\bfu)} {_1}j$.
It follows from $\ip(\bfu)=\ip(\bfv)$ that $\bfv=\bfv_1z\bfv_2$ such that $z\not \in \con(\bfv_1)$ and $\con(\bfv_1)=\con(\bfu_1)$.
This implies that ${_1}i\prec_{\phi(\bfv)} {_1}j$.
Therefore for any $1\leq i<j\leq 4$, ${_1}i\prec_{\phi(\bfu)} {_1}j$ if and only if  ${_1}i\prec_{\phi(\bfv)} {_1}j$.
\end{proof}

An immediate consequence of Theorem \ref{lem:inv3ids4} is the following.

\begin{corollary}
The decision problem ${\textsc{Check-Id}}(\sn,^{\sigma})$ for each finite $n\geq 4$
belongs to the complexity class $\mathsf{P}$.
\end{corollary}

\begin{theorem}\label{thm:inv3FBS4}
The variety generated by  $(\s_4,^{\sigma})$ is defined by the  identities \eqref{id: inv} and \eqref{id:xhxyty}.
\end{theorem}
\begin{proof}
Clearly, $(\s_4,^{\sigma})$ satisfies the identities \eqref{id: inv} and \eqref{id:xhxyty} by Theorem~\ref{lem:inv3ids4}.
It suffices to show that each non-trivial identity satisfied by $(\s_4,^{\sigma})$ can be deduced from  \eqref{id:xhxyty}.
Let $\Sigma$ be the set of all non-trivial balanced identities satisfied by $(\s_4,^{\sigma})$ but  not  deducible from \eqref{id:xhxyty}. Suppose that $\Sigma\neq \emptyset$. Note that any balanced identity $\bfu\approx \bfv$ in $\Sigma$ can be written
uniquely in the form $\bfu' a\bfw \approx \bfv' b\bfw$ where $a,b$ are distinct variables and $|\bfu'|=|\bfv'|$. Choose an identity, say $\bfu\approx \bfv$, from $\Sigma$ such that when it is written as $\bfu' a\bfw\approx \bfv' b\bfw$,
the lengths of the words $\bfu'$ and $\bfv'$ are as short as possible. Since $\bfu\approx \bfv$ is balanced, we have $\con(\bfu'a)=\con(\bfv'b)$ and $\occ(x, \bfu'a)=\occ(x, \bfv'b)$ for any $x \in \con(\bfu'a)$.

Since $\con(\bfu'a)=\con(\bfv'b)$, we have $b\in \con(\bfu')$, and so $\bfu'a=\bfu_1 bc\bfu_2$ with $b\neq c$ and $b\not\in\con(\bfu_2)$.  Since $\con(\bfu'a)=\con(\bfv'b)$, we have $c\in \con(\bfv')$, and so $\bfv'=\bfv_1 c\bfv_2$ with $c\not\in\con(\bfv_2)$. Thus $\bfu=\bfu_1 bc\bfu_2\bfw$ and $\bfv=\bfv_1 c\bfv_2b\bfw$.
On the one hand, we claim that $b\in\con(\bfu_1)$ or $c\in\con(\bfu_1)$. If $b, c \not\in\con(\bfu_1)$, then $\occ(b, \bfu'a)=1$, and so $\bfu[b,c]$ starts with $b$ but $\bfv[b,c]$ does not, which contradicts (ii) of Theorem~\ref{lem:inv3ids4}.
On the other hand,
we claim that $b\in\con(\bfu_2\bfw)$ or $c\in\con(\bfu_2\bfw)$. If $b, c\not\in\con(\bfu_2\bfw)$, then $\bfu[b,c]$ ends with $c$ but $\bfv[b,c]$ does not, which contradicts (ii) of Theorem~\ref{lem:inv3ids4}.
Therefore, the
identities \eqref{id:xhxyty} can be used to convert $\bfu_1 bc\bfu_2 \bfw$ into $\bfu_1cb\bfu_2\bfw$. By repeating this process, the word $\bfu=\bfu'a\bfw=\bfu_1bc\bfu_2 \bfw$ can be converted into the word $\bfu_1c\bfu_2b \bfw$ by the
identities \eqref{id:xhxyty}.

Clearly the identities $\bfu_1c\bfu_2b\bfw \approx \bfu'a\bfw \approx \bfv'b\bfw$ hold in $(\s_4, ^{\sigma})$. Note that $|\bfu_1c\bfu_2b\bfw|=|\bfu'a\bfw|=|\bfv'b\bfw|$ and words in the identity $\bfu_1c\bfu_2b\bfw \approx \bfv'b\bfw$ have a longer common suffix than words in the identity $\bfu'a\bfw \approx \bfv'b\bfw$. Hence $\bfu_1c\bfu_2b\bfw \approx \bfv'b\bfw \notin \Sigma$ by the minimality assumption on the lengths of $\bfu'$ and $\bfv'$, that is, $\bfu_1c\bfu_2b\bfw \approx \bfv'b\bfw$ can be deducible from \eqref{id:xhxyty}. We have shown that $\bfu'a\bfw \approx \bfu_1c\bfu_2b\bfw$ can be deducible from \eqref{id:xhxyty}. Hence $\bfu'a\bfw \approx \bfv'b\bfw$ can be deducible from \eqref{id:xhxyty}, which contradicts $\bfu'a\bfw \approx \bfv'b\bfw \in \Sigma$. Therefore $\Sigma=\emptyset$.
\end{proof}

\section*{Acknowledgments}
The authors thank the anonymous referee for his/her careful reading of this paper and valuable suggestions,
especially the suggestion to give representations of $\s_n$ with involution by using the results of Cain et al.

\end{sloppypar}
\end{document}